\documentclass[reqno]{amsart}
\usepackage{amsmath,amsthm}
\usepackage{amssymb,latexsym}
\usepackage{graphicx,subfig}
\usepackage{epsfig}
\usepackage{color}
\usepackage{epsfig}
\usepackage{amsaddr}

\numberwithin{equation}{section}

\usepackage{lineno}
\newcommand*\patchAmsMathEnvironmentForLineno[1]{%
  \expandafter\let\csname old#1\expandafter\endcsname\csname #1\endcsname
  \expandafter\let\csname oldend#1\expandafter\endcsname\csname end#1\endcsname
  \renewenvironment{#1}%
     {\linenomath\csname old#1\endcsname}%
     {\csname oldend#1\endcsname\endlinenomath}}%
\newcommand*\patchBothAmsMathEnvironmentsForLineno[1]{%
  \patchAmsMathEnvironmentForLineno{#1}%
  \patchAmsMathEnvironmentForLineno{#1*}}%
\AtBeginDocument{%
\patchBothAmsMathEnvironmentsForLineno{equation}%
\patchBothAmsMathEnvironmentsForLineno{align}%
\patchBothAmsMathEnvironmentsForLineno{flalign}%
\patchBothAmsMathEnvironmentsForLineno{alignat}%
\patchBothAmsMathEnvironmentsForLineno{gather}%
\patchBothAmsMathEnvironmentsForLineno{multline}%
}

\newtheorem{thm}{Theorem}[section]
\newtheorem{dfn}[thm]{Definition}
\newtheorem{lma}[thm]{Lemma}
\newtheorem{cor}[thm]{Corollary}
\newtheorem{prp}[thm]{Proposition}

\newtheorem{clm}[thm]{Claim}

\newtheorem{fact}[thm]{Fact}

\def\bad{\text{bad}}
\def\eps{\varepsilon}

\title{Perfect matchings in $3$-partite $3$-uniform hypergraphs}

\author{Allan Lo}
\address{School of Mathematics, University of Birmingham, Birmingham, B15~2TT, UK}
\email{s.a.lo@bham.ac.uk}
\thanks{A.~Lo was supported by the  European Research Council
under the European Union's Seventh Framework Programme (FP/2007--2013) / ERC Grant
Agreement n. 258345.}

\author{Klas Markstr{\"o}m}
\address{Department of Mathematics and Mathematical Statistics, Ume\r{a} University, S-901 87 Ume\r{a}, Sweden}
\email{klas.markstrom@math.umu.se}

\date{\today}

\begin{document}

\begin{abstract}
Let $H$ be a $3$-partite $3$-uniform hypergraph, i.e. a $3$-uniform hypergraph such that every edge intersects every partition class in exactly one vertex, with each partition class of size $n$.
We determine a Dirac-type vertex degree threshold for perfect matchings in $3$-partite $3$-uniform hypergraphs.
\end{abstract}

\maketitle


\section{Introduction} \label{sec:introduction}
A perfect matching in a graph $G$ is a set of vertex-disjoint edges, which covers all vertices of~$G$.
Tutte~\cite{MR0023048} gave a characterization of all graphs that contain a perfect matching.
An easy consequence of a celebrated theorem of Dirac~\cite{MR0047308} is that if $G$ is a graph of even order $n$ and minimum degree $\delta(G) \ge n/2$, then $G$ contains a perfect matching.
Thus, it is natural to ask for Dirac-type degree thresholds for perfect matchings in hypergraphs.

We follow the notation of~\cite{MR1633290} and denote by $\binom{U}k$ the set of all $k$-element subsets of a set $U$.
We will often write a $k$-set to mean a $k$-element set.
A \emph{$k$-uniform hypergraph}, or \emph{$k$-graph} for short, is a pair $H = (V(H),E(H))$, where $V(H)$ is a finite set of vertices and the edge set $E(H)$ is a set of $k$-subsets of $V(H)$.
Often we write $V$ instead of $V(H)$ when it is clear from the context.
A matching $M$ in $H$ is a set of vertex-disjoint edges of~$H$, and it is \emph{perfect} if $M$ covers all vertices of~$H$.
Clearly, a perfect matching only exists if $|V|$ is divisible by~$k$.

Given a $k$-graph $H$ and an $l$-set $T \in \binom{V}l$, let $\deg(T)$ be the number of $(k-l)$-sets $S \in \binom{V}{k-l}$ such that $S \cup T$ is an edge in~$H$. 
Let $\delta_l(H)$ be the \emph{minimum $l$-degree of $H$}, that is, $\min \deg(T)$ over all $T \in \binom{V}l$.
We define $m_l(k,n)$ to be the smallest integer $m$ such that every $k$-graph $H$ of order $n$ satisfying $\delta_l(H) \ge m$ contains a perfect matching.
Hence, we always assume that $k|n$ whenever we talk about $m_l(k,n)$.
Thus we have $m_1(2,n) = n/2$, by the result of Dirac.

For $k \ge 3$ and $l=k-1$, R{\"o}dl, Ruci{\'n}ski and Szemer{\'e}di~\cite{MR2500161} determined the value of $m_{k-1}(k,n)$ exactly, which improved the bound given in~\cite{MR2207573}.
For $k \ge 3$ and $1 \le l < k$, it is conjectured in~\cite{MR2496914} that 
\begin{align}
m_l(k,n) \sim \max \left\{ \frac12, 1- \left( 1- \frac{1}k\right)^{k-l} \right\} \binom{n}{k-l}. \label{eqn:m_l(k,n)}
\end{align}
For $k=3$ and $l = 1$, H{\`a}n, Person and Schacht~\cite{MR2496914} showed that \eqref{eqn:m_l(k,n)} is true, that is, $m_{1}(3,n) \sim \frac59 \binom{n}2$ improving on a result of Daykin and H{\"a}ggkvist~\cite{MR615135} for $k=3$.
The exact value was independently determined by Khan~\cite{2011arXiv1101.5830K} and K{\"u}hn, Osthus and Treglown~\cite{kuhn2010matchings}.
Khan~\cite{2011arXiv1101.5675K} further determined $m_1(4,n)$ exactly.
For $k \ge 3$ and $k/2 \le l < k$, Pikhurko~\cite{MR2438870} proved that $m_l(k,n) \sim \frac12 \binom{n}{k-l}$.
Recently, exact values of $m_l(k,n)$ for all $k/2 \le l < k$ were determined by Czygrinow and Kamat~\cite{MR2928635} and by Treglown and Zhao~\cite{MR2925939, TreglownZhao2}.
Alon, Frankl, Huang, R{\"{o}}dl, Ruci\'{n}ski and Sudakov~\cite{alon2011large} determined the asymptotic value of $m_l(k,n)$ when $k-l \le 4$.
Thus, for $1 \le l < k/2$, \eqref{eqn:m_l(k,n)} is still open except for a few cases.
Partial results were proved by H{\`a}n, Person and Schacht~\cite{MR2496914} and later improved by the second author and Ruci{\'n}ski~\cite{markstromperfect}.
We recommend~\cite{rodldirac} for a survey of other results on perfect matchings in hypergraphs.

Instead of seeking a perfect matching, Bollob\'{a}s, Daykin and Erd\H{o}s~\cite{MR0412030} considered Dirac-type degree thresholds for a matching of size~$m$.
\begin{thm}[Bollob\'{a}s, Daykin and Erd\H{o}s~\cite{MR0412030}] \label{thm:BollobasDaykinErdos}
Let $k$ and $m$ be integers with $k \ge 2$.
If $H$ is a $k$-graph of order~$n \ge 2k^3(m+2)$ and
$$\delta_1(H)  > \binom{n-1}{k-1} - \binom{n-m}{k-1},$$ then $H$ contains a matching of size $m$.
\end{thm}

For $k=3$, K{\"u}hn, Osthus and Treglown~\cite{kuhn2010matchings} extended the above result and proved that if $\delta_1(H)  > \binom{n-1}{2} - \binom{n-m}{2}$ and $m \le n/3$ (provided that $n$ is sufficiently large), then $H$ contains a matching of size $m$.
Moreover, this result and Theorem~\ref{thm:BollobasDaykinErdos} are best possible.

\subsection{Matchings in $k$-partite $k$-graphs}

For $k \in \mathbb{N}$, we refer to the set $\{1,2, \dots k\}$ as $[k]$.
A $k$-graph $H$ is \emph{$k$-partite}, if there exists a partition of the vertex set $V$ into $k$ classes $V_1, \dots, V_k$ such that every edge intersects each class in exactly one vertex.
We say that $H$ is \emph{balanced} if $|V_1| = |V_2| = \dots = |V_k|$.
Clearly, a perfect matching can only exist if $H$ is balanced.
Given a $k$-partite $k$-graph $H$ and an integer $l \in [k]$, an $l$-set $T \in \binom{V}l$ is said to be \emph{legal} if $| T \cap V_i| \le 1$ for $i \in [k]$.
Let $\delta_l(H)= \min \deg(T)$ over all legal $l$-sets in~$H$.
We define $m'_l(k,n)$ to be the smallest integer $m$ such that every $k$-partite $k$-graph $H$ with $n$ vertices in each class satisfying $\delta_l(H) \ge m$ contains a perfect matching.
Note that we no longer assume that $k|n$ for $m'_l(k,n)$.
Aharoni, Georgakopoulos and Spr{\"u}ssel~\cite{MR2460215} proved that $m'_{k-1}(k,n) \le n/2+1$.
Pikhurko~\cite{MR2438870} proved an Ore-type result for perfect matchings in $k$-partite $k$-graphs.
Given an $l$-set $L \in \binom{[k]}l$, an $l$-set $T \in \binom{V}{l}$ is an \emph{$L$-tuple} if $|T \cap V_i| =1$ for all $i \in L$ and so $|T \cap V_j| =0$ for $j \notin L$.
Let $\delta_L(H) = \min \deg(T)$ over all $L$-tuples~$T$.

\begin{thm}[Pikhurko~\cite{MR2438870}]
Let $1 \le l < k$ and $L \in \binom{[k]}l$.
Let $H$ be a $k$-partite $k$-graph with partition classes $V_1, \dots, V_k$ each of size $n$.
If 
\begin{align}
	\frac{\delta_{L}(H)}{n^{k-l}} + \frac{\delta_{[k]\setminus L}(H)}{n^l} \ge 1+o(1) \nonumber
\end{align}
then $H$ contains a perfect matching.
\end{thm}

This implies that for $ k/2 \le l <k$, $m'_l(k,n) \sim \frac12 n^{k-l}$.
In this paper, we determine $m'_1(3,n)$, that is the minimum $\delta_1(H)$ that ensures a perfect matching in a $3$-partite $3$-graph~$H$.
First we bound $m'_1(3,n)$ from below by considering the following examples.

\begin{dfn}[${H}_k(n;d_1, \dots, d_k), \mathcal{H}_k(n;m;d),H_k(n;m)$] \label{dfn:H_k(n;m)}
Let $V_1,\dots, V_k$ be disjoint vertex sets with $|V_i| = n$ for all $i \in [k]$.
For all $i \in [k]$, let $U_i$ and $W_i$ be a partition of $V_i$ with $|W_i| = d_i$.
Define $H_k(n;d_1, \dots, d_k)$ to be the $k$-partite $k$-graph with partition classes $ V_1, \dots, V_k$ consisting of all those edges which intersect $W = \bigcup_{i \in [k]} W_i$.
Define $\mathcal{H}_k(n;m;d)$ to be the family of $k$-partite $k$-graphs $H(n;d_1, \dots, d_k)$ with $\max d_i = d$ and $\sum_{i \in [k]} d_i = m$.
Define $H_k(n;m)$ to be $H_k(n;d_1, \dots, d_k)$ with $d_i =\lfloor (m +i-1)/k \rfloor$ for all $i \in [k]$.
\end{dfn}

Notice that every matching in $H_3(n;m)$ has size at most $\min \{m, n\}$, since every edge intersects $\bigcup_{i \in [k]} W_i$ and $\sum_{i \in [k]} |W_i| = m$.
Also,
\begin{align} \nonumber
	 \delta_1(H_3(n;m) ) = n^2 - \left( n-\lfloor m/3 \rfloor \right)(n-\lfloor (m+1)/3 \rfloor).
\end{align}
This suggests that $m'_1(k,n) > \delta_1(H_k(n;n-1))$.
We also consider the next example.
A family $\mathcal{A}$ of sets is \emph{intersecting} if $A \cap B \ne \emptyset$ for all $A,B \in \mathcal{A}$.

\begin{dfn}[$\mathcal{H}^*_k(n;m)$] \label{dfn:H^*_k(n;m)}
Define $\mathcal{H}^*_k(n;m)$ to be the family of $k$-partite $k$-graphs such that
\begin{align}
\nonumber
	\mathcal{H}^*_k(n;m) = \{H_k(n;m-1) \cup H' : \text{$E(H')$ is an intersecting family } \},
\end{align}
where $H'$ is also a $k$-partite $k$-graph on $V(H_k(n;m-1))$.
\end{dfn}

Note that every matching in $H^* \in \mathcal{H}^*_3(n;m)$ also has size at most $\min \{m, n\}$.
Given $H_3(n;m-1)$, we fix $u_i \in U_i$ for each $i \in [3]$ and let $H'$ be the $3$-partite $3$-graph such that $E(H')$ is the set of all legal $3$-sets $T$ such that $|T \cap \{u_1, u_2, u_3\}| \ge 2$.
Note that $E(H')$ is an intersecting family.
Hence, the $3$-graph $H^*_3(n;m) = H_3(n;m-1) \cup H'$ is a member of $\mathcal{H}^*_3(n;m)$ and $\delta_1(H^*_3(n;m)) = \delta_1(H_3(n;m-1))+1$.

Given integers $n \ge m \ge 1$, let $d_3(n,m) = \max \{\delta_1(H_3(n;m) ), \delta_1(H_3(n;m-1) )+1\}$.
Explicitly 
\begin{align*}
d_3(n,m) =  \begin{cases}
			n^2 - \left( n-\lfloor m/3 \rfloor \right)(n-\lfloor (m+1)/3 \rfloor) & \text{if $m \ne 1 \pmod{3}$,}\\
			n^2 - \left( n- (m-1)/3 \right)^2+1	& \textrm{if $m = 1 \pmod{3}$.}
		\end{cases}
\end{align*} 
Hence, $m'_1(3,n) > d_3(n,n-1)$ by considering $H_3(n;n-1)$ and $H^*_3(n;n-1)$.
We show that this bound is sharp for sufficiently large $n$, that is, $m'_1(3,n) = d_3(n,n-1)+1$.

\begin{thm} \label{thm:exact}
There is an integer $n_0$ such that for all $n \ge n_0$, $m'_1(3,n) = d_3(n,n-1)+1$.
\end{thm}

In addition, we also prove a natural generalization of Theorem~\ref{thm:BollobasDaykinErdos} for $k$-partite $k$-graphs~$H$, that is, a Dirac-type $\delta_1(H)$ threshold for a matching of size $m$ in $k$-partite $k$-graphs~$H$.
Let $\delta_l(\mathcal{H}_k(n;m;d)) = \min \{ \delta_l(H) : H \in \mathcal{H}_k(n;m;d)\}$.
We show that if $H$ has no matching of size $m+1$ and $\delta_1(H) \ge \delta_1(\mathcal{H}_k(n;m;\lceil m/k \rceil))$ for sufficiently large $n$, then $H$ is a subgraph of $H' \in \mathcal{H}_k(n;m;\lceil m/k \rceil) \cup \mathcal{H}^*_k(n;m)$.

\begin{thm} \label{thm:partialmatching}
Let $k$, $m$ and $n$ be integers such that $k \ge 2$ and $n \ge k^7 m$.
Let $H$ be a $k$-partite $k$-graph with $n$ vertices in each class.
Suppose the largest matching in $H$ is of size~$m$ and $$\delta_1(H) \ge \delta_1( \mathcal{H}_k(n;m;\lceil m/k \rceil)) \approx (m-\lceil m /k \rceil -o(1) )n^{k-2}.$$
Then $H$ is a subgraph of $H' \in \mathcal{H}_k(n;m;\lceil m/k \rceil) \cup \mathcal{H}^*_k(n; m )$.
Moreover, if $m \ne 1 \pmod{k}$, then $H$ is a subgraph of  $H' \in \mathcal{H}_k(n;m;\lceil m/k \rceil)$.
\end{thm}

In particular, for $k=3$, we deduce the following corollary.

\begin{cor} \label{cor:partialmatching}
Let $H$ be a $3$-partite $3$-graph with each class of size $n \ge 3^7m $ and $\delta_1(H) > d_3(n,m)$.
Then $H$ contains a matching of size $m+1$.
\end{cor}

Again, the bound given in the corollary above is optimal, as can be seen by considering $H_3(n;m)$ and $H^*_3(n;m)$.
Hence, we ask whether every $3$-partite $3$-graph $H$ with each class of size $n > m $ and $\delta_1(H) > d_3(n,m)$ contains a matching of size~$m+1$.

The layout of the paper is as follows.
In the next section, we set up some basic notation and establish a few facts about $\mathcal{H}_k(n;m;d)$ and $\mathcal{H}^*_k(n;m)$.
In Section~\ref{sec:partialmatching}, we prove Theorem~\ref{thm:partialmatching} and Corollary~\ref{cor:partialmatching}.
The remainder of the paper is dedicated to the proof of Theorem~\ref{thm:exact}.

Here, we give an outline of the proof of Theorem~\ref{thm:exact}, which uses the absorption technique of R\"{o}dl, Ruci\'{n}ski and Szemer\'{e}di~\cite{MR2500161}.
The corresponding $k$-partite version of the absorption lemma for $k$-graphs, Lemma~\ref{lma:absorptionlemma}, is proven in Section~\ref{sec:absorption}.
Let $H$ be a $3$-partite $3$-graph with $\delta_1(H) > d_3(n,n-1)$.
The absorption lemma implies that there exists a small matching $M$ such that for every `small balanced' set $W \subseteq V(H) \setminus V(M)$ there exists a perfect matching in $H[V(M) \cup W]$.
Thus, to prove Theorem~\ref{thm:exact}, it is sufficient to show that there exists a matching covering almost all vertices in $H' = H \setminus V(M)$ as the remaining vertices can then be `absorbed' by $M$ to get a perfect matching in $H$.
New ideas are needed to overcome the constraints imposed by $H'$ being $3$-partite.
Rather than getting bogged down in the details and calculations, we first prove that Theorem~\ref{thm:exact} holds asymptotically, Theorem~\ref{thm:asymptotic}, in Section~\ref{sec:asymptotic} to setup the framework and illustrate these ideas.
We then refine these arguments in Section~\ref{sec:epsilon-close} to show that either $H'$ contains a large matching (in which case we are done by absorption) or $H$ `looks like' the extremal graph $H'_3(n;n)$ (defined in Section~\ref{sec:3partite}).
Finally,  in Section~\ref{sec:extremalresult}, we show that in the latter case $H$ contains a perfect matching by an application of Corollary~\ref{cor:partialmatching}.

\section{Notation and Extremal graphs} \label{sec:notation}

For $a,b \in \mathbb{N}$, we refer to the set $\{a,a+1, \dots ,b\}$ as $[a,b]$.
Let $H$ be a $k$-partite $k$-graph with vertex classes $V_1, \dots, V_k$.
For integers $m$ with $k |m$, we say that a vertex set $W \subseteq V$ is a \emph{balanced $m$-set} if $|W \cap V_i| =m/k$ for all $i \in [k]$.
Throughout this paper, unless stated otherwise, we define $v_{i,j}$ to be the vertex in $ V_i \cap V(e_j)$ for a partition class $V_i$ and an edge $e_j$.

Given a $k$-graph $H$ and a vertex set $U \subseteq V(H)$, we denote by $H[U]$ the subgraph of $H$ induced by~$U$.
Given $k$-graphs $H$ and $H'$, we write $H - H'$ to denote the subgraph of $H$ obtained by removing all the edges in $E(H) \cap E(H')$.

\subsection{Properties of $\mathcal{H}_k(n;m;d)$ and $ \mathcal{H}^*_k(n;m)$}
By the definition of $H(n;d_1,\dots, d_k)$, we have
\begin{align}
	\delta_l(H_k(n;d_1,\dots, d_k)) = n^{k-l} - \max_{I \in \binom{[k]}{k-l}} \left\{ \prod_{i \in I} \left(n- d_i \right) \right\}\nonumber
\end{align}
for $1 \le l <k$.
Moreover, 
\begin{align}
	\delta_l(H_k(n;m)) = n^{k-l} - \prod_{i \in [k-l]} \left(n- \left\lfloor \frac{m+i-1}k \right\rfloor \right) \nonumber
\end{align}
and by concavity
\begin{align}
	\delta_l(H_k(n;m))	= \max\{ \delta_l(H) : H \in \mathcal{H}_k(n;m;\lceil m/k \rceil)  \} .\label{eqn:maxd3(n,m)}
\end{align}
Let $l = 1$ and let $0 \le d_1 \le  \dots \le d_k = d$ be integers.
Set $m = \sum d_i$.
Further suppose that $n \ge  (k-1)d$.
Notice that 
\begin{align}
& \delta_1 \left( H_k(n;d_1, \dots, d_k) \right)  \nonumber \\
& = n^{k-1} - \prod_{i \in [k-1]} \left(n- d_i \right) 
 = \sum_{1 \le j \le k-1}  (-1)^{j+1} n^{k-1-j} \sum_{I \in \binom{[k-1]}{j}}   \prod_{i \in I} d_i \nonumber
\\
& \ge n^{k-2} \sum_{i \in [k-1]} d_i  - \sum_{ 2 \le j \le k-1}   n^{k-1-j} \sum_{I \in \binom{[k-1]}{j}}   \prod_{i \in I} d_i \nonumber\\
& \ge (m- d )n^{k-2} - \sum_{ 2 \le j \le k-1} \binom{k-1}j n^{k-1-j} d^{j}\nonumber\\
& \ge (m- d )n^{k-2} - (k-1)^2 d^2 n^{k-3}\sum_{ 2 \le j \le k-1} \frac{1}{j!} \left(\frac{(k-1)d}{n} \right)^{j-2}\nonumber\\
& \ge (m- d )n^{k-2} - (k-1)^2 d^2 n^{k-3}, \label{eqn:d(H_k(n,m,d)}
\end{align}
where the last inequality holds since $n \ge (k-1)d$ and $\sum_{ j \ge 2 } (j!)^{-1} < \sum_{ j \ge 1 } 2^{-j} \le 1 $.

The proposition below gives a recursive relationship for $\delta_1(\mathcal{H}_k(n;m;d ))$, which will be used in the proof of Theorem~\ref{thm:partialmatching}.

\begin{prp} \label{prp:d(H_k)}
Let $d$, $k$, $m$ and $n$ be integers with $n \ge m > (k-1)d$.
Then \begin{align*}
	 \delta_1 ( \mathcal{H}_k(n-1;m-k;d-1 )) & =  \delta_1(\mathcal{H}_k(n;m;d ))-n^{k-1} + (n-1)^{k-1}.
\end{align*}
\end{prp}

\begin{proof}
Let $d_1, \dots, d_k$ be positive integers such that $\max d_i = d$ and $\sum d_i = m$.
Since $m > (k-1)d$, we have $d_i \ge 1$ for each $i \in[k]$.
Thus, there is a natural one-to-one relationship between $\mathcal{H}_k(n;m;d)$ and $\mathcal{H}_k(n-1;m-k;d-1)$, namely, $H_k(n;d_1, \dots, d_k) \in \mathcal{H}_k(n;m;d)$ maps to $H_k(n-1;d_1-1, \dots, d_k-1) \in \mathcal{H}_k(n-1;m-k;d-1)$.
Without loss of generality, we may assume that $d_k = d$.
Hence
\begin{align*}
 \delta_1 (H_k(n;d_1,\dots, d_k)) & =   n^{k-1} -  \prod_{i \in [k-1]} \left(n - d_i \right) 
 =  n^{k-1} -  \prod_{i \in [k-1]} \left(n -1 - (d_i-1) \right) \\
 & =  \delta_1(H_k(n-1;d_1-1,\dots, d_k-1)) + n^{k-1} - (n-1)^{k-1}.
\end{align*}
Thus, the proposition follows after taking the minimum over all such $d_i$'s.
\end{proof}

\section{Partial matchings} \label{sec:partialmatching}

Here, we prove the analogue of Theorem~\ref{thm:BollobasDaykinErdos} for $k$-partite $k$-graphs, Theorem~\ref{thm:partialmatching}.
For $k=2$, the theorem can easily be verified by K\"{o}nig's theorem and the minimum degree condition.
Also it is trivial for $m=1$, so we may assume that $k \ge 3$ and $m \ge 2$ in this section.
The proof divides into two main parts.
If $n \ge 4k^3m^2$, then we analyze the structure of $H$ directly (Lemma~\ref{lma:partialmatching}).
If $k^7 m \le n \le 4k^3 m^2$, then we proceed by induction on $n$, where its proof is based on~\cite{MR0412030}.

We need the following definitions, which are only used in this section.
Let $V_1, \dots, V_k$ be partition classes each of size $n$.
Let $M = \{ e_1, \dots, e_m\}$ be a matching of size $m$.
Recall that $\{v_{i,j}\} = V_i \cap V(e_j)$ for $i \in [k]$ and  $e_j \in M$.
Given $i \in [k]$, $x \in V \setminus V(M)$ and $e_j \in M$, we say that $x$ and $e_j$ are \emph{$i$-connected with respect to $M$} if there exist more than $2k^2 n^{k-3}$ edges containing both $x$ and $v_{i,j}$ but no other vertices in~$M$.
In addition, we say that $x$ makes \emph{$l$ connections to a submatching $M' \subseteq M$ with respect to $M$} if there exist $l$ distinct pairs $(e_j,i_j)$ such that $x$ and $e_j$ are $i_j$-connected (with respect to $M$) for $e_j \in M'$ and $i_j \in [k]$.
The following lemma studies the properties of $i$-connectness.

\begin{lma} \label{lma:connect}
Let $k$, $l$, $m$ and $n$ be integers such that $k \ge 3$ and $l \le 2k, m \le n$.
Let $H$ be a $k$-partite $k$-graph with $n$ vertices in each class.
Suppose $M = \{ e_1, \dots, e_m\}$ is a matching of size $m$ in $H$.
\begin{enumerate}
	\item[(i)] Let $U \subseteq V$ with $|U| < 2k$ and $U \cap V(M) = \emptyset$.
	Further assume that there exist distinct vertices $x_1, \dots, x_l \in V \setminus (U \cup V(M))$, where $x_j$ is $i_j$-connected to $e_j$ for $j \in [l]$. 
	Then there exists a matching $\{f_1, \dots, f_l\}$ with $V(f_j) \cap (V(M) \cup U \cup \{x_1, \dots, x_l\}) = \{v_{i_j,j}, x_j\}$ for $j \in [l]$.
	In particular, $M \cup \bigcup_{j \in [l]}f_j - \bigcup_{j \in [l]}e_j$ is a matching of size~$m$ disjoint from $U$.
	\item[(ii)] If $e_j \in M$ is $i$- and $i'$-connected to $x$ and $x'$ respectively for distinct $x,x' \in V \setminus V(M)$ and $i \ne i'$, then there is a matching of size $m+1$.
\end{enumerate}
\end{lma}

\begin{proof}
First we prove (i).
For each $j \in[l]$, we are going to choose an edge $f_j$ containing $x_j$ and $v_{i_j,j}$ but no other vertices in $U_j$, where $U_j = V(M) \cup U \cup \bigcup_{j' \in [j-1]} V(f_{j'})\cup \{x_{j}, \dots, x_l\}$.
Suppose we have already chosen $f_1, \dots, f_{j-1}$ and we choose $f_j$ as follows.
Without loss of generality, $x_j \in V_1$ and $i_j =2$.
Set $U'_j = U_j \setminus (V(M) \cup V_1 \cup V_2)$.
Note that 
\begin{align*}
|U'_j| & \le |U| + \sum_{j' \in [j-1]} |V(f_{j'}) \setminus (V_1 \cup V_2)| +  |\{x_{j+1}, \dots, x_l\}|\\
& \le  2k + (j-1)(k-2) + l-j \le 2k^2.
\end{align*}
The number of edges containing $x_j$, $v_{i_j,j}$ and a vertex in $U'_j$ is at most $|U'_j| n^{k-3} \le 2k^2 n^{k-3}$.
Since $x_j$ and $e_j$ are $2$-connected and $U'_j \cap V(M)= \emptyset$, there is an edge $f_j$ containing $x_j$ and $v_{2,j}$ but no other vertices in $U_j' \cup V(M) = U_j$.
Thus, there exist vertex-disjoint edges $f_1, \dots, f_l$ as desired.
Also, $M \cup \bigcup_{j\in[l]}f_j -\bigcup_{j\in[l]} e_j $ is a matching of size $m$ as desired.

We now deduce (ii) from (i).
The number of edges containing $x$, $x'$ and $v_{i,j}$ is at most $n^{k-3}$.
Since $x$ is $i$-connected to $e_j$, there exists an edge~$f$ containing both $x$ and $v_{i,j}$ but no other vertices in $V(M) \cup \{x'\}$.
Set $U = V(f) \setminus \{ v_{i,j} \}$.
By~(i) taking $x_1 = x'$, there exists an edge $f'$ containing $x'$ with $V(f') \cap V(M \cup f) = \{v_{i',j}\}$.
Thus, $M \cup \{f,f'\}-e_j$ is a matching of size $m+1$.
\end{proof}

By~\eqref{eqn:d(H_k(n,m,d)}, the following lemma implies Theorem~\ref{thm:partialmatching} for $n \ge  4k^3m^2$.
We give a rough sketch of its proof when $m \ne 1 \pmod{k}$.
Let $M$ be a matching of maximal size in $H$.
Since $\delta_1(H)$ is large, each $x \in X$ makes many connections to~$M$, where $X = V \setminus V(M)$.
For each $i \in [k]$, let $M_i$ be the set of edges in $M$ that are $i$-connected to at least one $x \in X$.
We then deduce that the set of $M_1, \dots, M_k$ forms a partition of~$M$.
By setting $W_i = V(M_i) \cap V_i$ for each $i \in [k]$, we show that $H$ is a subgraph of $H_k(n;d_1, \dots, d_k)$ with $ d_i = |W_i| $.

\begin{lma} \label{lma:partialmatching}
Let $k$, $m$ and $n$ be integers such that $k \ge 3$, $m \ge 2$ and $n \ge 4k^3m^2$.
Let $H$ be a $k$-partite $k$-graph with $n$ vertices in each class.
Suppose the largest matching in $H$ is of size $m$ and 
\begin{align*}
\delta_1(H) \ge (m-\lceil m/k \rceil)n^{k-2} - m^2 n^{k-3}.
\end{align*}
Then $H$ is a subgraph of $H' \in \mathcal{H}_k(n;m;\lceil m/k \rceil) \cup \mathcal{H}^*_k(n; m )$.
Moreover, if $m \ne 1 \pmod{k}$, then $H$ is a subgraph of  $H' \in \mathcal{H}_k(n;m;\lceil m/k \rceil)$.
\end{lma}

\begin{proof}
Let $V_1, \dots, V_k$ be the partition classes of $H$.
Let $M = \{e_1, \dots, e_m \}$ be a matching in~$H$ with $m$ maximal.
Let $r$ and $s$ be the unique integers such that $m = rk+s$, $r \ge 0$ and $1 \le s \le k$.
This implies that
\begin{align}
\delta_1(H) \ge (m-r-1)n^{k-2} - m^2 n^{k-3} \ge (m-r-5/4) n^{k-2} \nonumber
\end{align}
as $n \ge 4k^3m^2$.
Set $X_i = V_i \setminus V(M)$ for each $i \in[k]$ and $X = \bigcup_{i \in [k]} X_i$.
For $x \in X$, every edge containing $x$ must intersect with $V(M)$ by the maximality of $M$.
Hence, the number of edges containing $x$ and exactly one vertex of $M$ is at least
\begin{align}
	 \deg(x)- k^2 m^{2} n^{k-3}
	& \ge \delta_1(H)-  n^{k-2}/4
	 & \ge \left(m-r-3/2 \right)n^{k-2}. \label{eqn:partialL_x(M)}
\end{align}
Therefore, $x$ makes at least $m-r-1$ connections to $M$.
Otherwise, by~\eqref{eqn:partialL_x(M)} we have 
\begin{align}
	(m-r-3/2)n^{k-2} 
	& \le  mk \times 2k^2n^{k-3} + (m-r-2)n^{k-2}.  \nonumber
 \nonumber
\end{align}
Since $n \ge 4k^3 m^2 \ge  8k^3 m $, the right hand side of the inequality above is at most $(m-r-7/4) n^{k-2}$ implying a contradiction.

We say that an edge $e_j \in M$ is \emph{bad}, if there exists a vertex $x \in X$ such that $e_j$ is $i$-connected to $x$ for two distinct values of $i \in [k]$.
We say that a vertex $x \in X$ is \emph{bad} if $x$ has two connections with some $e \in M$.
Hence, each bad edge is connected to a bad vertex.
By Lemma~\ref{lma:connect}(ii) and the maximality of $M$, if $e_j \in M$ is $i$- and $i'$-connected to $x$ and $x'$ respectively for $x,x' \in X$, then $i = i'$ or $x = x'$.
This implies that if $e_j$ is bad, then $e_j$ is connected to precisely one vertex $x \in X$ and moreover this vertex is bad.
Thus, there are at most $m$ bad vertices.
Set $X'_i$ be the set of vertices in $X_i$ that are not bad and $X' = \bigcup_{i \in [k]} X_i'$.
So $|X'_i| \ge n -m - m >0$.

Given $i \in [k]$ and $x \in X'$, denote by $M_i(x)$ the set of edges $e \in M$ such that $x$ and $e$ are $i$-connected.
Let $M_i$ be the set of edges in $M$ that are $i$-connected to some $x \in X'$.
Note that $M_i \cap M_{i'} = \emptyset$ for all $i \ne i'$ by Lemma~\ref{lma:connect}(ii) and the maximality of~$M$.
Also, for $x \in X_i'$,
\begin{align}
\text{$M_i(x) = \emptyset$  and $M_j(x) \subseteq M_j$ for all $j \in [k]$.} \label{Mi(x)}
\end{align}
Recall that $x$ makes at least $m-r-1$ connections to $M$, so 
\begin{align}
	\sum_{j \in [k] \setminus \{i\}} |M_j(x)| = \sum_{j \in [k]} |M_j(x)|  \ge m-r-1. \label{Mi(x)2}
\end{align}
For $i \in [k]$ and $x \in X'_i$, we have
\begin{align}
m = & |M|  \ge   |M_i|+\sum_{j \in [k] \setminus \{i\}} |M_{j}(x)| \ge |M_i| + m-r-1, \label{eqn:|M'_1|}\\
|M_i| & \le r+1 \label{eqn:|M'_1|2}.
\end{align}
We now divide into two cases depending on whether $s \ne 1$ (i.e. $m \ne 1 \pmod{k}$) or not.

\noindent\textbf{Case 1: $s \ne 1$.}
Note that $2 \le s \le k$.
Pick $x_k \in X'_k$.
Since $m = rk +s \ge rk +2$, \eqref{Mi(x)2} implies that $\sum_{j \in [k-1]} |M_{j}(x_k)| \ge r(k-1)+1$.
By \eqref{Mi(x)} and~\eqref{eqn:|M'_1|2}, we have $| M_j (x_k) | \le |M_j| \le r+1$ for all $j \in [k]$. 
Without loss of generality, we may assume that $|M_1(x_k)| = r+1$.
Hence, $M_1(x_k) = M_1$.

Next pick $x_1 \in X'_1$.
By taking $x = x_1$ and $i = 1$, we must have equality in~\eqref{eqn:|M'_1|} implying that $M_i(x_1) = M_i$ for all $i \in [k] \setminus \{1\}$ (as $M_i(x_1) \subseteq M_i$ by~\eqref{Mi(x)}).
Therefore, $M_1$, \dots,  $M_k$ partition $M$ with $M_i(x_1) \cup M_i(x_k) = M_i$.
Recall that $x_1$ and $x_k$ are not bad.
Since every edge $e \in M$ is connected to $x_1$ or $x_k$, $M$ contains no bad edge and so there is no bad vertex in $X$.
Thus $X_i = X'_i$ for all $i \in [k]$.

Recall that $|M| = rk+s \ge rk+2$ and $|M_i| \le r+1$ by~\eqref{eqn:|M'_1|2}.
Note that vertices $x_1$ and $x_k$ are no longer needed in our argument, so we can relabel the $M_i$ if necessary such that $|M_1| = |M_2| = r+1$.
Hence, $\sum_{i \in [k] \setminus \{j\}}|M_i| = m- r-1$ for $j =1,2$.
By~\eqref{Mi(x)} and~\eqref{Mi(x)2}, every vertex $x \in X_1 \cup X_2$ makes exactly $m-r-1$ connections to $M$.
Therefore, for $j = 1,2$ and for all $i \in [k] \setminus \{j\}$, every $x \in X_j$ is $i$-connected to every edge $e \in M_i$.

Set $W_i = V(M_i) \cap V_i$ for $i \in [k]$ and $W = \bigcup_{i \in [k]} W_i$.
Let $d_i = |M_i| = |W_i|$, so by~\eqref{eqn:|M'_1|2} $d_i \le r+1$ for all $i \in [k]$ and $d_1 = d_2 = r+1$.
We are going to show that $H \subseteq H_k(n;d_1, \dots, d_k)$, which then implies the lemma (as $\sum_{i \in [k]} d_i = m$).
It is sufficient to show that all edges in $H$ meet~$W$.
Suppose the contrary and let $e'$ be an edge vertex-disjoint from~$W$.
Clearly $e'$ must intersect with some edge $e \in M$ by the maximality of~$M$.
Let $e_1, \dots, e_l$ be the edges in $M$ with $V(e') \cap V(e_j) \ne \emptyset$ and $e_j \in M_{i_j}$ for $j \in [l]$.
So $l \le k$.
Next, we pick distinct vertices $x_1, \dots, x_l \in X \setminus V(e')$ such that $x_j \in X_1$ if $i_j \ne 1$ otherwise $x_j \in X_2$.
Set $U = V(e') \setminus V(M)$, so $|U| < k$.
There exists a matching $\{f_1, \dots, f_l\}$ satisfying condition~(i) of Lemma~\ref{lma:connect}.
Hence, $M  \cup \{ e', f_1, \dots, f_l \} -\bigcup_{j=1}^{l} e_j $ is a matching of size $m+1$, a contradiction.

\noindent\textbf{Case 2: $s = 1$.}
Define $M_i'$ to be the set of edges in $M$ that are $i$-connected to at least $4k$ vertices in $X$.
Thus $M_i' \subseteq M_i$ for all $i \in [k]$.
Each edge $e \in M'_i$ makes at most $|X \setminus X_i| = (k-1)(n-m)$ connections to $X$.
Since each $x \in X$ makes at least $m-r-1$ connections to $M$, we have $\sum_{i \in [k]} |M'_i| \ge m-1$, or else
\begin{align*}
	(m-r-1) |X| & \le (m-2) \times (k-1)(n-m) + 2 \times 4k,
\end{align*}
a contradiction as $|X| = k(n-m)$.
If $M = \bigcup_{i \in [k]} M'_i$, then set $W_i = V(M'_i) \cap V_i$ for each $i \le [k]$ and so $\max |M_i| = r+1$ by~\eqref{eqn:|M'_1|2} and $m=kr+1$.
By using a similar argument as used in Case~1, we deduce that $H$ is a subgraph of $H_k(n;d_1, \dots,d_k)$, where $d_i = |W_i|$.
Hence, we may assume that $|M \setminus \bigcup M'_i| = 1$.
If $|M'_1| = r+1$ say, then by~\eqref{eqn:|M'_1|} each $x \in X_1'$ is connected to every $e \in M_i$ for all $i \in [k] \setminus \{1\}$.
Since $|X_1'| \ge n-2m \ge 4k$, we have $M = \bigcup M_i'$, a contradiction.
Therefore, we may assume that $|M'_i| = r$ for all $i \in [k]$.
Let $M' = \bigcup M'_i$.
Set $W_i = V(M'_i) \cap V_i$ for $i \in [k]$ and $W = \bigcup_{i \in [k]} W_i$.

Let $H'$ be the subgraph of $H$ induced by the edges not intersecting with~$W$.
Suppose there exist two vertex-disjoint edges $e'_1$ and~$e'_2$ in $H'$.
Notice that $e'_1$ or $e'_2$ must intersect with $M'$ or else $M' \cup \{ e'_1, e'_2\}$ is a matching of size $m+1$, a contradiction. 
Set $U = V(e'_1 \cup e'_2) \setminus V(M')$, so $|U| <  2k$.
Let $e_1, \dots, e_l$ be the edges in $M'$ that intersect with $e'_1$ or $e'_2$ with $e_j \in M'_{i_j}$.
Note that $l \le  2k$.
Since each $e_j$ is $i_j$-connected to $4k$ vertices in $X$, there exist distinct vertices $x_1, \dots, x_l \in X \setminus U$ with $x_j$ is $i_j$-connected to $e_j$.
Therefore, there exists a matching $\{f_1, \dots, f_l\}$ satisfying the condition~(i) of Lemma~\ref{lma:connect} taking $M = M'$.
Observe that $M'  \cup \{e'_1 ,e'_2 ,f_1, \dots, f_l \} -\bigcup_{j=1}^{l} e_j $ is a matching of size $m+1$, a contradiction.
Thus, $E(H')$ is an intersecting family. 
Since $|W_i| = r$ for every $i \in [k]$, $H$ is a subgraph of $H_k(n;m-1) \cup H'$.
Moreover, $H_k(n;m-1) \cup H' \in \mathcal{H}^*_k(n;m)$ and so the proof of the lemma is completed.
\end{proof}

By Lemma~\ref{lma:partialmatching}, to prove Theorem~\ref{thm:partialmatching}, it remains to consider the case $k^7 m \le n \le 4k^3 m^2$ and so $m \ge k^4/4$.
By~\eqref{eqn:d(H_k(n,m,d)} and the fact that $n \ge k^7 m $, we have
\begin{align*}
\delta_1(\mathcal{H}_k(n;m;\lceil m/k \rceil)) & \ge  (m-\lceil m/k \rceil)n^{k-2}- (k-1)^2 \lceil m/k \rceil^2 n^{k-3}\\
	& \ge  \left( (k-1)/k - 1/k^3 \right) m n^{k-2}.
\end{align*}
The next lemma shows that there exists a vertex in each partition class such that its vertex degree is large.
Its proof uses the ideas in the proof of Lemma~\ref{lma:partialmatching}.

\begin{lma} \label{lma:partialmathcinginduction}
Let $k$, $m$ and $n$ be integers such that $k \ge 3$, $m \ge 2$ and $n \ge k^7 m$.
Let $H$ be a $k$-partite $k$-graph with $n$ vertices in each class.
Suppose the largest matching in $H$ is of size $m$ and 
$\delta_1(H) \ge \left( \frac{k-1}k - \frac{1}{k^3}\right) m n^{k-2}$.
Then there exists a vertex $v$ in each partition class with $\deg(v) \ge n^{k-1}/4k$.
\end{lma}

\begin{proof}
Let $V_1, \dots, V_k$ be the partition classes of $H$ and let $M = \{e_1, \dots, e_m\}$ be a largest matching of size $m$ in $H$.
Let $X_i = V_i \setminus V(M)$ for $i \in [k]$ and $X = \bigcup_{i \in [k]} X_i$.
For $x \in X$, the number of edges containing $x$ and exactly one vertex of $M$ is at least
\begin{align}
	 \delta_1(H)-\binom{k-1}2 m^{2} n^{k-3} & \ge \left(\frac{k-1}k  - \frac{1}{k^3}  - \frac{1}{k^5} \right) m n^{k-2}\label{eqn:partialL_x(M)induction}
\end{align}
as $n \ge k^7 m $.
Recall that $x$ and $e_j$ are $i$-connected if there exist more than $2k^2 n^{k-3}$ edges containing both $x$ and $v_{i,j}$ but no other vertices in~$M$.
Since
\begin{align*}
	\left(\frac{k-1}k  - \frac{1}{k^3}  - \frac{1}{k^5} \right) m n^{k-2} \ge 2k^2 n^{k-3} m+
\left(\frac{k-1}k - \frac{2}{k^3}  \right) m n^{k-2},
\end{align*}
every $x \in X$ makes at least $\left(\frac{k-1}k - \frac{2}{k^3}   \right) m$ connections to~$M$.
In addition, if an edge $e_j \in M$ is $i$- and $i'$-connected to vertices $x$ and $x' \in X$ respectively, then $i = i'$ or $x = x'$ by Lemma~\ref{lma:connect}(ii) and the maximality of $M$.
Thus, there are at most $m$ vertices $x \in X$ that have two connections with some $e \in M$.
Remove these vertices from $X_i$ and call the resulting set $X'_i$ and let $X' = \bigcup_{i \in [k]} X'_i$.
Note that $|X'_i| \ge n-2m>0 $ for each $i \in [k]$.

Given $i \in [k]$, define $M_i$ to be the set of edges $e_j \in M$ such that $e_j$ is $i$-connected to at least one vertex $x \in X'$.
If there is an edge $e \in M$ that is not connected to any vertex $x \in X'$, then we arbitrarily assign $e$ to exactly one of~$M_i$.
Thus, $M_1, \dots, M_k$ partition~$M$.
For $x \in X'$, denote by $M_i(x)$ the set of edges $e \in M$ such that $x$ and $e$ are $i$-connected.
Hence, for $x \in X_i$
\begin{align}
\text{$M_i(x) = \emptyset$  and $M_j(x) \subseteq M_j$ for all $j \in [k]$.} \label{Mi(x)3}
\end{align}
Recall that $x$ makes at least $\left(\frac{k-1}{k} - \frac{2}{k^3} \right) m$ connections to $M$ and at most one connection to each edge of $M$, so 
\begin{align}
	\sum_{j \in [k] \setminus \{i\}} |M_j(x)|  = \sum_{j \in [k]} |M_j(x)|\ge \left( \frac{k-1}{k} - \frac{2}{k^3}\right) m. \label{Mi(x)4}
\end{align}
Hence, \eqref{Mi(x)3} and \eqref{Mi(x)4} imply that
\begin{align}
	m & \ge  |M_i| + \sum_{j \in [k] \setminus \{i\} } |M_j(x)| \ge |M_i| + \left( (k-1)/k - 2/k^3  \right) m, \nonumber \\
	|M_i| & \le \left(  1/k + 2/k^3  \right)m .	\label{eqn:M_i}
\end{align}
Set $W_i = V_i \cap M_i$ and $W = \bigcup_{i \in [k]} W_i$.

\begin{clm}
There are at least $m n^{k-1}/(2k^2)$ edges that meet~$W_1$.
\end{clm}
\begin{proof}[Proof of Claim]
Recall that if $x \in X$ is not $i$-connected to $e_j \in M$, then there are at most $2 k^2 n^{k-3}$ edges containing $x$ and $v_{i,j} = V_i \cap e_j$ but no other vertices in $M$.
Pick $x \in X'_2$.
Since $x$ is not connected to any edges in $M_2$, there are at most $k |M_2| \times 2 k^2 n^{k-3} = 2k^3n^{k-3}|M_2|$ edges containing $x$ and exactly one vertex of $V(M_2)$ but no other vertices in~$M$. 
Also, $x$ makes at most one connection to each edge in~$M$.
Together with Lemma~\ref{lma:connect}(ii) and the maximality of~$M$, $x$ is not $i'$-connected to any edges in $M_i$ for all $i \ne i'$.
Hence, there are at most $(k-1) (|M | - |M_2|) \times 2 k^2 n^{k-3}$ edges containing $x$ and exactly one vertex of $V(M)\setminus (V(M_2) \cup W)$ but no other vertices in~$M$.
Note that there is no edge containing $x$ and a vertex in~$W_2$ (as $x \in V_2$ and $W_2 \subseteq V_2$).
In total, the number of edges containing $x$ and exactly one vertex of $(V(M) \setminus W) \cup W_2$ but no other vertices in~$M$ is at most
\begin{align}
	  2k^3 n^{k-3}|M_2| + (k-1)(|M | - |M_2|) \times 2 k^2 n^{k-3} 
	\le 2 k^3 m n^{k-3} 
	\le 2 mn^{k-2}/k^4. \nonumber
\end{align}
The number of edges containing $x$ and at least one vertex of $W \setminus (W_1\cup W_2)$ is at most
\begin{align}
	& (|W| - |W_1| - |W_2|) n^{k-2} 
= \sum_{3 \le i \le k } |W_i| n^{k-2} 
= \sum_{3 \le i \le k } |M_i| n^{k-2} \nonumber \\
\le &  (k-2)\left(   1/k + 2/ k^3 \right)m n^{k-2} \nonumber
\end{align}
by~\eqref{eqn:M_i}.
In addition, there are at most $k^2 m^2 n^{k-3} \le m n^{k-2}/k^5$ edges containing $x$ and at least two vertices of $M$.
By the maximality of $M$, every edge containing $x$ must meet $V(M)$.
Thus, the number of edges containing $x$ and exactly one vertex of $W_1$ is at least
\begin{align}
	& \delta_1(H)
	- {2 mn^{k-2}}/{k^4}
	- (k-2)\left( {1}/k + {2}/{k^3} \right)m n^{k-2}
	- {mn^{k-2}}/{k^5}\nonumber 
	 \ge  {mn^{k-2}}/{k^2}. \nonumber
\end{align}
Since $x \in X'_2$ is chosen arbitrarily and $|X'_2|  \ge n-2m \ge n/2$, there are at least $m n^{k-1}/(2k^2)$ edges that meet $W_1$ as claimed.
\end{proof}
Recall that $|W_1| = |M_1| \le 2 m/k$ by~\eqref{eqn:M_i}, so there exists a vertex $v_1 \in W_1$ such that $\deg(v_1) \ge n^{k-1}/4k$.
By similar arguments, there exists $v_i \in W_i \subseteq V_i$ such that $\deg(v_i) \ge n^{k-1}/4k$ for each $i \in [k]$ as required.
\end{proof}

We are now ready to prove Theorem~\ref{thm:partialmatching}.
We need the following simple proposition of which we omit the proof.

\begin{prp} \label{prp:easymindegree}
Let $H$ be a $k$-partite $k$-graph with $n$ vertices in each class.
Suppose the largest matching in $H$ is of size $m$.
\begin{itemize}
	\item[(a)] Given a vertex $v \in V(H)$, if $H \setminus v$ contains a matching of size $m$ for some vertex $v$, then $\deg(v) \le n^{k-1} - (n-m)^{k-1}$. 

\item[(b)] Given a legal $k$-set $T$, if $H \setminus T$ contains a matching of size $m-k+1$, then there exists a vertex $v \in T$ with $\deg(v) \le n^{k-1} - (n-m)^{k-1}$.
\end{itemize}
\end{prp}

\begin{proof} [Proof of Theorem~\ref{thm:partialmatching}]
For $k=2$, the theorem can be easily verified by K\"{o}nig's theorem and the minimum degree condition.
Also it is trivial for $m=1$, so we may assume that $k \ge 3$ and $m \ge 2$.
We proceed by induction on $n$ and $m$.
Let $V_1, \dots, V_k$ be the partition classes of $H$.
First suppose that $m \le k^4/4$.
Note that $n \ge k^7 m \ge 4k^3 m^2$ and \eqref{eqn:d(H_k(n,m,d)} implies that 
\begin{align*}
\delta_1( \mathcal{H}_k(n;m;\lceil m/k \rceil)) 
& \ge (m- \lceil m/k \rceil )n^{k-2} - (k-1)^2\lceil m/k \rceil^2 n^{k-3} \\
& \ge (m- \lceil m/k \rceil )n^{k-2}  -m^2 n^{k-3}.
\end{align*}
So Theorem~\ref{thm:partialmatching} is true by Lemma~\ref{lma:partialmatching}.

Therefore, we may assume that $m > k^4/4$.
Let $T = \{ v_1, \dots, v_k\}$ with $v_i \in V_i$ and $\deg(v_i) \ge n^{k-1}/4k$.
Note that $T$ exists by Lemma~\ref{lma:partialmathcinginduction} as $\delta_1(H) \ge \left( (k-1)/k - 1/k^3 \right)m n^{k-2}$ (see the calculation before Lemma~\ref{lma:partialmathcinginduction}).
Let $H' = H \setminus T$, so
\begin{align}
\delta_1(H') 
	& \ge   \delta_1(\mathcal{H}_k(n;m;\lceil m/k \rceil ) ) - (n^{k-1} - (n-1)^{k-1}) \nonumber\\
	& =  \delta_1(\mathcal{H}_k(n-1;m-k;\lceil (m-k)/k \rceil ) ) \nonumber
\end{align}
where the last equality is due to Proposition~\ref{prp:d(H_k)}.
Assume that $m \ne 1 \pmod{k}$.
If there does not exist a matching of size greater than $m-k$ in~$H'$, then the induction hypothesis implies that $H'$ is a subgraph of $H'' \in  \mathcal{H}_k(n-1;m-k;\lceil (m-k)/k \rceil )$.
This means that $H$ is a subgraph of a member of $\mathcal{H}_k(n;m;\lceil m/k \rceil )$.
Thus, we may assume that $H'$ contains a matching of size $m-k+1$.
By Proposition~\ref{prp:easymindegree}(b), there exists a vertex $v_i \in T$ such that $\deg (v_i) \le n^{k-1} - (n-m)^{k-1} \le kmn^{k-2}$ as $n$ is large.
Since $v_i \in T$, $\deg (v_i) \ge n^{k-1}/4k$ implying that $n \le  4k^2m$, a contradiction.
A similar argument holds for the case when $m=1 \pmod{k}$.
\end{proof}

Next, we prove Corollary~\ref{cor:partialmatching}.

\begin{proof}[Proof of Corollary~\ref{cor:partialmatching}]
Suppose that $H$ is a $3$-partite $3$-graph with $\delta_1(H) > d_3(n,m)$.
First assume that $m \ne 1 \pmod3$ and so $d_3(n,m) = \delta_1(H_3(n;m))$.
By Theorem~\ref{thm:partialmatching}, $H$ either contains a matching of size $m+1$ or is a subgraph of $H' \in \mathcal{H}_3(n;m)$. 
Note that \eqref{eqn:maxd3(n,m)} implies that $\delta_1(H) > \delta_1(H')$ for all $H' \in \mathcal{H}_3(n;m)$.
Therefore, $H$ contains a matching of size $m+1$ and so the corollary is true for $m \ne 1 \pmod{3}$.

Thus, we may assume that $m=1 \pmod{3}$.
Suppose that $H$ does not contain a matching of size $m+1$, so Theorem~\ref{thm:partialmatching} implies that $H$ is a subgraph of $H' \in \mathcal{H}_3(n;m) \cup \mathcal{H}^*_3(n;m)$.
By~\eqref{eqn:maxd3(n,m)} and the fact that $\delta_1(H) > \delta_1(H_3(n;m))$, we deduce that $H$ is a subgraph of $H' \in  \mathcal{H}^*_3(n;m)$.
Recall that $d_3(n,m) = \delta_1(H_3(n;m-1))+1$ and the definition of $\mathcal{H}^*_3(n;m)$, Definition~\ref{dfn:H^*_k(n;m)}.
Therefore, to prove the corollary, it is enough to show that if $H''$ is a $3$-partite $3$-graph with $V(H'') = V(H)$ and the edge set of $H''$ forms an intersecting family, then $\delta_1(H'') \le 1$.
Let $V_1$, $V_2$ and $V_3$ be the partition classes of $H''$.
Suppose the contrary, $E(H'')$ is an intersecting family with $\delta_1(H'') \ge 2$.

First we claim that there exist two edges $e_1,e_2 \in E(H'')$ with $e_1 = x_1 x_2 u$, $e_2 = y_1 y_2 u$, distinct $x_i, y_i \in V_i$ and $u \in V_3$.
Pick $ u \in V_3$. 
Since $\deg(u) \ge 2$, there exists two edges containing $u$.
If these two edges only intersect at $u$, then the claim holds.
Without loss of generality, suppose $x_1 x_2 u$ and $x_1' x_2 u$ are the two edges with $x_1, x'_1 \in V_1$, $x_2 \in V_2$ and $u \in V_3$.
Let $y_2 \in V_2 \setminus x_2$.
Since $E(H'')$ is intersecting, all edges incident with $y_2$ must contains~$u$.
Note that $\deg(y_2) \ge 1$, so there exists an edge $e_2$ containing both $y_2$ and $u$.
Without loss of generality (by relabelling $x_1$ and $x_1'$ if necessary), we may assume that $e_2 = y_1 y_2 u$ with $y_1 \ne x_1$.
Therefore, the claim holds.

Let $w \in V_3 \setminus \{u\}$.
Recall that $\deg(w) \ge \delta_1(H'') \ge 2$ and $E(H'')$ is an intersecting family.
Therefore, $e_3 = x_1y_2w$ and $e_4 = y_1x_2w$ are edges.
However, $\{e_1,e_2,e_3,e_4,e\}$ is not an intersecting family for any edges $e$ satisfying $V(e) \cap \{u,w\} = \emptyset$, which exist as $|V_3| \ge 3$ and $\delta_1(H'') \ge 2$.
This is a contradiction.
\end{proof}

\section{An absorption lemma for $k$-partite $k$-graphs} \label{sec:absorption}

Here, we prove a $k$-partite version of the absorption lemma for $k$-graphs given by H{\`a}n, Person and Schacht~\cite{MR2496914}.
Its proof follows along the same lines as in~\cite{MR2496914}.
For completeness, we include the proof below.
First we need the following simple proposition.

\begin{prp} \label{prp:delta_l}
Let $H$ be a $k$-partite $k$-graph with $n$ vertices in each class.
For all $0 \le a \le 1$ and all integers $1 \le m \le l < k$, if $	\delta_l (H) \ge a n^{k-l} \textrm{, then } \delta_m (H) \ge a n^{k-m}$.
\end{prp}

\begin{proof}
Note that it suffices to prove the case when $m = l-1$. 
Let $V_1, \dots, V_k$ be the partition classes of $H$.
Let $T$ be a legal $(l-1)$-set in $V(H)$.
Without loss of generality, $T = \{v_1, \dots, v_{l-1}\}$ with $v_i \in V_i$ for $i \in [ l-1 ]$.
The condition on $\delta_l(H)$ implies that $T$ is contained in at least 
\begin{align}
	\sum_{v_l \in V_l} \deg( T \cup \{  v_l\} ) \ge n \delta_l(H) \ge a n^{k-l+1}. \nonumber
\end{align}
edges, and the proposition follows.
\end{proof}

\begin{lma}[An absorption lemma for $k$-partite $k$-graphs] \label{lma:absorptionlemma}
Let $1 \le l < k$, $0< \gamma < 1/(10k^3)$ and $\gamma' =\gamma^{2k-1}/20$.
Then there is an integer $n_0$ such that for all $n > n_0$ the following holds: Suppose $H$ is a $k$-partite $k$-graph with $n$ vertices in each class and minimum $l$-degree $\delta_l(H) \ge (1/2 + \gamma)n^{k-l}$, then there exists a matching $M$ in $H$ of size $|M| \le (k-1)\gamma^k n$ such that, for every balanced set $W$ of size $|W|  \le k \gamma' n$, there exists a matching covering exactly the vertices of $V(M) \cup W$.
\end{lma}

\begin{proof}
Let $H$ be a $k$-partite $k$-graph with partition classes $V_1, \dots, V_k$ each of size $n$ and minimum $l$-degree $\delta_l(H) \ge (1/2 + \gamma)n^{k-l}$.
From Proposition~\ref{prp:delta_l}, $\delta_1(H) \ge (1/2 + \gamma)n^{k-1}$ and it suffices to prove the lemma for $l=1$.
Throughout the proof we may assume that $n_0$ is chosen sufficiently large.
Furthermore set $m=k(k-1)$ and call a balanced $m$-set $A$ an \emph{absorbing $m$-set for} a balanced $k$-set $T$ if $A$ spans a matching of size $k-1$ and $A \cup T$ spans a matching of size $k$.
In other words, $A \cap T = \emptyset$ and both $H[A]$ and $H[A \cup T]$ contain perfect matchings.
Denote by $\mathcal{L}(T)$ the set of all absorbing $m$-sets for $T$.
Next, we show that for every balanced $k$-set $T$, there are many absorbing $m$-sets for $T$.

\begin{clm} \label{clm:numberofabsorbingm-set}
For every balanced $k$-set $T \subseteq V(H)$, $|\mathcal{L}(T)| \ge \gamma^{k-1} n^{m}/2((k-1)!)^{k}$.
\end{clm}

\begin{proof}
Let $T= \{v_1, \dots, v_k\}$ be fixed with $v_i \in V_i$ for $i \in[k]$.
Since $n$ was chosen large enough, there are at most $(k-1)n^{k-2} \le \gamma n^{k-1}$ edges, which contain $v_1$ and $v_j$ for some $j \in [2,k]$.
By $\delta_1(H)$ there are at least $n^{k-1}/2$ edges containing $v_1$ but none of $v_2, \dots, v_k$.
We fix one such edge $v_1 u_2 \dots u_k$ with $u_i \in V_i$ for $i \in [2,k]$.
Set $U_1 = \{ u_2, \dots, u_k\}$ and $W_0 = T$.
For each $j \in [2,k]$ and each pair $u_j, v_j$ suppose we succeed to choose a $(k-1)$-set $U_j$ such that $U_j$ is disjoint to $W_{j-1} = U_{j-1} \cup W_{j-2}$ and both $U_j \cup \{ u_j \}$ and $U_j \cup \{ v_j\}$ are edges in $H$.
Then for a fixed $j \in [2,k]$ we call such a choice $U_j$ \emph{good}, motivated by $A = \bigcup_{j \in [k]} U_j$ being an absorbing $m$-set for $T$.

Note that in each step $j \in [2,k]$ there are $k+(j-1)(k-1)$ vertices in $W_{j-1}$.
Moreover, there are at most $j \le k$ vertices in $V_i \cap W_{j-1}$ for all $i \in [k]$.
Thus, the number of edges intersecting $u_j$ (or $v_j$ respectively) and at least one other vertex in $W_j$ is at most $(k-1) j n^{k-2} < k^2n^{k-2} \le \gamma n^{k-1}$.
Note that there are at least $2 \gamma n^{k-1}$ sets $U'$ such that both $U' \cup \{ u_j \}$ and $U' \cup \{ v_j\}$ are edges in $H$.
Hence, for each $j \in[2,k]$, there are at least $2 \gamma n^{k-1} - \gamma n^{k-1}  =\gamma n^{k-1}$ good choices for $U_j$.
Thus, in total we obtain $\gamma^{k-1} n^{k(k-1)}/2$ absorbing $m$-sets for $T$ with multiplicity at most $((k-1)!)^{k}$.
\end{proof}

Now, choose a family $F$ of balanced $m$-sets by selecting each of the $\binom{n}{k-1}^k$ possible balanced $m$-sets independently with probability 
\begin{align}
	p = \frac{\gamma^k  ((k-1)!)^k} { 2 n^{m-1} }\le \frac{\gamma^k n} {2 \binom{n}{k-1}^k } .\label{eqn:probabiltyforabsorbinglemma}
\end{align}
Then by Chernoff's bound (see e.g.~\cite{MR1885388}) with probability $1-o(1)$ as $n \rightarrow \infty$, the family $F$ satisfies the following properties:
\begin{align}
|F| & \le  \gamma^k n\label{eqn:|F|}
\intertext{and}
|\mathcal{L}(T) \cap F| & \ge  p |\mathcal{L}(T)|/2 \ge \gamma^{2k-1} n/10 =  2 \gamma' n \label{eqn:L(T)}
\end{align} 
for all balanced $k$-sets $T\subseteq V(H)$.
Furthermore, we can bound the expected number of intersecting pairs of $m$-sets from above by
\begin{align}
	\binom{n}{k-1}^k \times k(k-1) \times \binom{n}{k-2}  \binom{n}{k-1}^{k-1} \times p^2 \le  \gamma^{2k} k^3 n/4 \le \gamma' n/2.
\nonumber
\end{align}
Thus, using Markov's inequality, we derive that with probability at least $1/2$
\begin{align}
	\textrm{$F$ contains at most $ \gamma' n$ intersecting pairs of $m$-sets.} \label{eqn:F}
\end{align}
Hence, with positive probability the family $F$ has all properties stated in \eqref{eqn:|F|}, \eqref{eqn:L(T)} and \eqref{eqn:F}.
Delete one $m$-set from each intersecting pair in such a family~$F$.
Further remove all non-absorbing $m$-sets, we get a subfamily~$F'$ consisting of pairwise disjoint balanced $m$-sets, which satisfies
\begin{align}
|\mathcal{L}(T) \cap F'| & \ge  2\gamma' n - \gamma' n \ge \gamma' n \nonumber
\end{align}
for all balanced $k$-sets $T$.
Since $F'$ consists only of absorbing $m$-sets, $H[V(F')]$ has a perfect matching $M$ of size at most $(k-1) \gamma^k n$ by~\eqref{eqn:|F|}.
For a balanced set $W \subseteq V \setminus V(M)$ of size $|W| \le k\gamma' n$, $W$ can be partition into at most $\gamma' n$ balanced $k$-sets.
For each balanced $k$-set $T$, we can successively choose a distinct absorbing $m$-set for $T$ in $F'$.
Hence, there exists a matching covering $V(M) \cup W$. 
\end{proof}

\section{Finding large matchings in $3$-partite $3$-graphs.} \label{sec:asymptotic}

Let $H$ be a $3$-partite $3$-graph and let $M$ be a matching satisfying the conditions of the absorption lemma, Lemma~\ref{lma:absorptionlemma}.
We remove the vertices of $M$ from $H$ and call the resulting graph $H'$.
Suppose that there exists a large matching $M_1$ in $H'$ covering almost all the vertices in $H'$. 
Let $W$ be the set of `leftover' vertices.
By the `absorption' property of $M$ there is a perfect matching $M_2$ in $H[V(M) \cup W]$.
Hence, $M_1 \cup M_2$ is a perfect matching in $H$.
Therefore, finding a large matching $M_1$ in $H'$ plays a crucial role in the proof of Theorem~\ref{thm:exact}.
However, since we do not have any control over the structure of $M$, $H'$ might not contain a large matching. 
If $H'$ does not contain a large matching, then we show that $H'$ contains a large subgraph, which `looks like' the extremal graph $H'_3(n';n')$ (defined in Section~\ref{sec:3partite}).

The aim of this section is to investigate the problem of finding large matchings in $3$-partite $3$-graphs. 
However, this problem is quite different from finding large matchings in $3$-graphs.
To this end, we introduce the concept of edge-coloured complete graphs, called matching graphs (defined in Section~\ref{sec:matchinggraphs}).
We study some of its properties, Lemmas~\ref{lma:forbiddencolouring} and~\ref{lma:app1}.
Using these results, we present a short proof showing that $m'_1(3,n) \sim 5n^2/9$, see Theorem~\ref{thm:asymptotic} below, which is the asymptotic version of Theorem~\ref{thm:exact}.

\begin{thm} \label{thm:asymptotic}
For all $0 < \gamma <10^{-6}$, there exists an integer $n_0 = n_0(\gamma)$ such that if $H$ is a $3$-partite $3$-graph with $n> n_0$ vertices in each class and $\delta_1(H) \ge \left( 5/9+ \gamma  \right) n^2$,
then $H$ contains a perfect matching.
\end{thm}

By further analyzing the structures of $H'$ and the matching graphs, we show that a $3$-partite $3$-graph $H$ with $\delta_1(H) \ge (5/9-\gamma)n$ either contain a large matching or `looks like' the extremal graph $H'_3(n';n')$, Lemma~\ref{lma:exact}. 

\subsection{Notations for $3$-partite $3$-graphs} \label{sec:3partite}

Now, we set up notations for $3$-partite $3$-graphs $H$ with partition classes, $V_1$, $V_2$, $V_3$, each of size $n$.
Given three vertex sets $U_1$, $U_2$, $U_3$, not necessarily disjoint, we say an edge $u_1u_2u_3$ is \emph{of type $U_1U_2U_3$} if $u_i \in U_i$ for all $i \in [3]$.

Recall the definition of $H(n;d_1,d_2, d_3)$ in Definition~\ref{dfn:H_k(n;m)}.
For $i \in [3]$, let $d_i \le n$ and define $H'(n;d_1,d_2, d_3)$ to be the resulting subgraph of $H(n;d_1,d_2, d_3)$ after removing all edges of type $WWW$.
In other words, $H'(n;d_1,d_2, d_3)$ is the $3$-partite $3$-graph with partition classes $U_1 \cup W_1$, $U_2 \cup W_2$, $U_3 \cup W_3$ consisting of all those edges of type $UUW$ and $UWW$, where $U = U_1 \cup U_2 \cup U_3$, $W = W_1 \cup W_2 \cup W_3$ and $|W_i| = d_i$ for all $i \in [3]$.
We write $H'_3(n;m)$ for $H'_3(n;\lfloor \frac{m}3 \rfloor, \lfloor \frac{m+1}3 \rfloor,\lfloor \frac{m+2}3 \rfloor)$.

Given $\varepsilon>0$, we say that $H$ is \emph{$\varepsilon$-close to $H'(n;d_1,d_2, d_3)$} if $|E(H'(n;d_1,d_2, d_3)-H)| \le \varepsilon n^3$.

The following lemma shows that if $\delta_1(H) \ge( 5/9-\gamma)n$ and $H$ does not contain a large matching, then $H$ contains a large subgraph, which is $\eps$-close to $H'_3(n';n')$. 
This lemma is proved in Section~\ref{sec:epsilon-close} by further analyzing the structure of $H$.

\begin{lma} \label{lma:exact}
Let $0< 4\rho < \gamma < 10^{-6}$ and $ \gamma'' = 3(1200 \gamma +\sqrt{\gamma/2})$.
There exists an integer $n_0$ such that if $H$ is a $3$-partite $3$-graph with $n>n_0$ vertices in each class and $\delta_1(H) \ge \left( 5/9 - \gamma \right) n^2$ and contains no matching of size $(1-\rho)n$, then there exists a subgraph $H'$ such that $H'$ is $8 \gamma''$-close to $H'_3(n';n')$, where $n' \ge (1- \gamma'')n$ and $3 |n'$.
\end{lma}

\subsection{Matching graphs} \label{sec:matchinggraphs}

Given a $3$-partite $3$-graph $H$ and a vertex $x \in V$, let $L_x$ be the \emph{link graph of $x$} such that the edge set of $L_x$ is precisely the set of all $2$-sets $T \subseteq V$ such that $\{ x \} \cup T$ is an edge in $H$.
Note that $L_x$ is a $2$-graph. 
Given disjoint vertex sets $U_1$ and $U_2$, let $L_x[U_1, U_2]$ be the bipartite subgraph of $L_x$ induced by partition classes $U_1$ and $U_2$.
Given a matching $M= \{e_1, \dots, e_p\}$, we write $L_x(e_1, \dots, e_p)$ for $\bigcup_{j \in [p-1]} L_x[V(e_j), V(e_{j+1})] $.
Let $L_x[M] = \bigcup_{(i,j) \in \binom{[p]}{2}} L_x[V(e_i), V(e_{j})] $.

Let $S$ be a triple $(x_1,x_2,x_3)$ with $x_i \in V_i$ for all $i \in [3]$.
The \emph{link graph $L_S$ of $S$} is defined to be the union of $L_{x_1}$, $L_{x_2}$ and $L_{x_3}$.
Note that for each $i \in[3]$, $L_{x_i}$ only has edges between $V_{i+1}$ and $V_{i+2}$ (addition modulo 3), so $L_S$ does not contain any multiple edges.
We define $L_S(e_1, \dots, e_p)$ and $L_S[M]$ analogusly for a matching $M= \{e_1, \dots, e_p\}$.

We say that a pair of vertex-disjoint edges $(e_1,e_2)$ is of \emph{type $(a_1,a_2,a_3)$ with respect to $S$} if $e(L_{x_i}(e_1,e_2)) = a_i$ for all $i\in[3]$.
By the definition of $L_{x_i}(e_1,e_2)$, we have $0 \le a_1,a_2,a_3 \le 2$.
Let $M$ be a matching in $H$.
Define the \emph{matching graph $G_S(M)$ with respect to $S$ and $M$} to be the edge-coloured complete graph with vertex set $M$ and for distinct $e_1,e_2 \in M$, the edge $(e_1,e_2)$ in $G_S(M)$ is coloured $(a_1,a_2,a_3)$ if and only if $(e_1,e_2)$ is of type $(a_1,a_2,a_3)$ with respect to $S$.
If $M$ is known from the context, then we simply write $G_S$ for $G_S(M)$.
From now on, all edge-coloured graphs are assumed have colours $(a_1,a_2,a_3)$ for $0 \le a_1,a_2,a_3\le 2$.
Define $c(vu)$ to be the colour of the edge $vu$.

\begin{dfn}[$p$-extensible] \label{dfn:ext}
Given an integer $p >0$, we say that an edge-coloured path $P = v_1 v_2\dots v_l$ is \emph{$p$-extensible} if the following statement holds.
Let $M = \{e_1, \dots, e_l\}$ be a matching of size $l$ and let $X_i = V_i \setminus V(M)$ for $i \in [3]$.
Let $S_1,S_2, \dots, S_p \in X_1 \times X_2 \times X_3$ be vertex-disjoint.
Suppose that for all $j \in [l-1]$ and for all $i,i' \in [p]$, $L_{S_i}(e_j,e_{j+1}) = L_{S_{i'}}(e_j,e_{j+1})$ and $(e_j,e_{j+1})$ is of type $c(v_jv_{j+1})$ with respect to $S_i$.
Then there exists a matching $M'$ of size $l+1$ in $H[V(M) \cup \bigcup_{i \in [p]} V(S_i)]$.
\end{dfn}

By the definition of $p$-extensible, we obtain the following fact, which we omit its proof.

\begin{fact} \label{fact:extensible}
Let $p$, $r$ and $s$ be integers with $r \le s$.
Let $P = v_1 \dots v_s$ be an edge-coloured path.
Let $H$ be a $3$-partite $3$-graph with partition classes $V_1$, $V_2$ and $V_3$.
Let $M = \{e_1, \dots, e_r\}$ be a matching in $H$ and $S = (x_1, x_2,x_3)$ with $x_i \in V_i \setminus V(M)$ such that the edge-coloured path $e_1 \dots e_r$ in $G_S(M)$ is an edge-coloured subpath of~$P$.
Further suppose that $L_S(e_1,\dots, e_r) = \bigcup_{j \in [r-1]} L_S(e_j,e_{j+1})$ contains a matching $M_0$ of size $r+1$ with at most $p$ edges between any two $V_i$'s.
Then $P$ is $p$-extensible.
\end{fact}

In the next lemma, we study a few short edge-coloured paths and their extensibilities.

\begin{lma} \label{lma:forbiddencolouring}
The following statements (and the corresponding statements obtained by swapping the indices) hold.
\begin{itemize}
\item[(i)] For $a_1 +a_2+a_3 \ge 5$, an edge with colour $(a_1,a_2,a_3)$ is $1$-extensible.  
\item[(ii)] A monochromatic path of length 5 with colour $(1,2,1)$ is $4$-extensible.
\item[(iii)] For integers $1 \le a_2 \le 2$ and $0 \le a'_1, a'_2 \le 2$, a path $P =v_1 \dots v_5$ of length $4$ such that $v_iv_{i+1}$ is coloured $(0,a_2,2)$ for $i \in [3]$ and $v_4v_5$ is coloured $(2,a'_1,a'_2)$ is $3$-extensible.
\item[(iv)] A path $P = v_1 \dots v_6$  of length $5$ such that $v_iv_{i+1}$ is coloured $(0,1,2)$ for $i \in [3]$ and $(1,0,2)$ for $i \in \{4,5\}$ is $5$-extensible.
\end{itemize}

\end{lma}

\begin{proof}
Let $H$ be a $3$-partite $3$-graph with partition classes $V_1$, $V_2$ and $V_3$.
For the rest of the proof, $M = \{e_1, \dots, e_r\}$ is assumed to be a matching of size~$r$ in~$H$ and $S = (x_1, x_2,x_3)$ with $x_i \in V_i \setminus V(M)$.
Recall that $\{v_{i,j}\} = V_i \cap e_j $.

(i)
Suppose that $(e_1,e_2)$ is of type $(a_1,a_2,a_3)$ with $a_1 +a_2+a_3 \ge 5$.
Notice by the definition of $L_S(e_1,e_2)$ that if $i=i'$ or $j = j'$ then $v_{i,j} v_{i',j'}$ is not an edge in $L_S(e_1,e_2)$.
Hence, there exists a matching $M_0 = \{f_1, f_2,f_3\}$ of size 3 in $L_S(e_1,e_2)$.
Since $|V_i \cap V(L_S(e_1,e_2))| =2$ for $i \in [3]$, without loss of generality we may assume that $V(f_i) \cap V_i = \emptyset$ for $i \in [3]$.
Hence (i) holds by Fact~\ref{fact:extensible}.

(ii)
Suppose the contrary that $P$ is not $4$-extensible.
Let $s =6$ and $P = e_1 e_2 \dots e_6$ be a monochromatic path of length 5 with colour $(1,2,1)$ in $G_S(M)$.
By Fact~\ref{fact:extensible}, we obtain a contradiction if $L_S(e_1,e_2)$ contains a matching of size~$3$.
Thus, we may assume that $L_S(e_1,e_2)$ is a path of length~$4$.
To be precise, $L_S(e_1,e_2) = v_{1,2}v_{3,1}v_{2,2}v_{1,1}v_{3,2}$ or $ v_{1,1}v_{3,2}v_{2,1}v_{1,2}v_{3,1}$.
We write $\overrightarrow{e_1e_2}$ if $L_S(e_1,e_2) = v_{1,2}v_{3,1}v_{2,2}v_{1,1}v_{3,2}$, and  write $\overrightarrow{e_2e_1}$ otherwise.
Similar statements also hold for $L_S(e_i,e_{i+1})$ for $i \in [5]$.
Thus, we may assume that $P$ is oriented (but may not be directed).
In the claim below, we show that $P$ does not contain any of the following oriented subpaths.

\begin{clm} \label{clm:forbiddencolouring1}
If $P$ is not $4$-extensible, then $P$ does not contain
\begin{itemize}
	\item[(a)] a path $u_1u_2u_3u_4$ with $\overleftarrow{u_1u_2}$, $\overleftarrow{u_2u_3}$ and $\overleftarrow{u_3u_4}$;
	\item[(b)] a path $u_1u_2u_3u_4$ with $\overrightarrow{u_{1}u_2}$, $\overrightarrow{u_{2}u_3}$ and $\overleftarrow{u_{3}u_4}$;
	\item[(c)]  a path $u_1u_2u_3u_4u_5$ with $\overrightarrow{u_{1}u_2}$, $\overleftarrow{u_{2}u_3}$, $\overrightarrow{u_{3}u_4}$ and $\overleftarrow{u_{4}u_5}$;
	\item[(d)] a path $u_1u_2u_3u_4u_5$ with $\overrightarrow{u_{1}u_2}$, $\overleftarrow{u_{2}u_3}$, $\overrightarrow{u_{3}u_4}$ and $\overrightarrow{u_{4}u_5}$.
\end{itemize}
\end{clm}

\begin{proof}[Proof of claim]
(a) Suppose we have edges $e_1, \dots, e_4$ with $\overleftarrow{e_1e_2}$, $\overleftarrow{e_2e_3}$ and $\overleftarrow{e_3e_4}$.
There exists a matching of size 5 in $L_{S}(e_1,e_2,e_3,e_4)$, namely $\{v_{1,1} v_{3,2}$, $v_{3,1} v_{1,2}$, $v_{2,2} v_{1,3}$, $v_{2,3} v_{3,4}$, $v_{3,3} v_{1,4} \}$ (see Figure~\ref{fig:clm:e_S(2,1,1)}~$(A)$), and so $P$ is $4$-extensible by Fact~\ref{fact:extensible}, a contradicition.

(b) Suppose we have edges $e_1, \dots, e_4$ with $\overrightarrow{e_1e_2}$, $\overrightarrow{e_2e_3}$, $\overleftarrow{e_3e_4}$.
There exists a matching of size 5 in $L_{S}(e_1,e_2,e_3,e_4)$, namely $\{v_{1,1} v_{2,2}$, $v_{3,1} v_{1,2}$, $v_{3,2} v_{1,3}$, $v_{2,3} v_{3,4}$, $v_{3,3} v_{1,4} \}$ (see Figure~\ref{fig:clm:e_S(2,1,1)}~$(B)$), and so $P$ is $4$-extensible by Fact~\ref{fact:extensible}, a contradicition.

(c) Suppose we have edges $e_1, \dots, e_5$ with $\overrightarrow{e_1e_2}$, $\overleftarrow{e_2e_3}$, $\overrightarrow{e_3e_4}$ and $\overleftarrow{e_4e_5}$.
Note that there is a matching of size 6 in $L_{S}(e_1,e_2,e_3,e_4,e_5)$, namely $\{v_{1,1} v_{3,2}$, $v_{3,1} v_{1,2}$, $v_{2,2} v_{1,3}$, $v_{3,3} v_{1,4}$, $v_{3,4} v_{1,5}$, $v_{2,4}v_{3,5}\}$ (see Figure~\ref{fig:clm:e_S(2,1,1)}~$(C)$), and so $P$ is $4$-extensible by Fact~\ref{fact:extensible}, a contradicition.

(d) Suppose we have edges $e_1, \dots, e_5$ with $\overrightarrow{e_1e_2}$, $\overleftarrow{e_2e_3}$, $\overrightarrow{e_3e_4}$ and $\overrightarrow{e_4e_5}$.
Note that there is a matching of size 6 in $L_{S}(e_1,e_2,e_3,e_4,e_5)$, namely $\{v_{1,1} v_{3,2}$, $v_{3,1} v_{1,2}$, $v_{2,2} v_{1,3}$, $v_{3,3} v_{2,4}$, $v_{1,4} v_{3,5}$, $v_{3,4}v_{1,5}\}$ (see Figure~\ref{fig:clm:e_S(2,1,1)}~$(D)$), and so $P$ is $4$-extensible by Fact~\ref{fact:extensible}, a contradicition.
\begin{figure}[tbp]
\centering
\subfloat[]{
\includegraphics[scale=0.6]{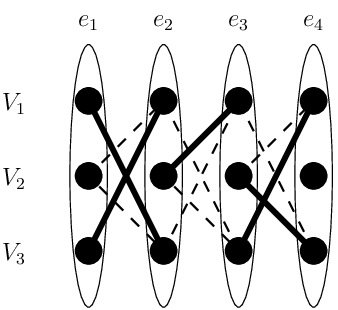}}
\subfloat[]{
\includegraphics[scale=0.6]{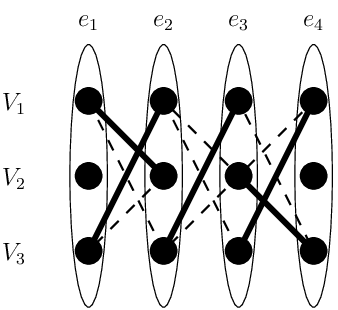}}

\subfloat[]{
\includegraphics[scale=0.6]{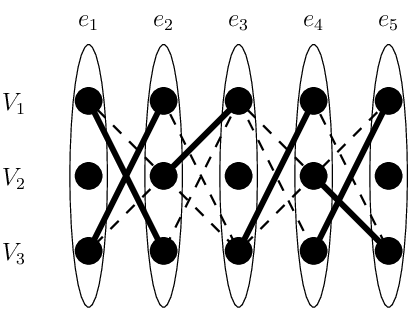}}
\subfloat[]{
\includegraphics[scale=0.6]{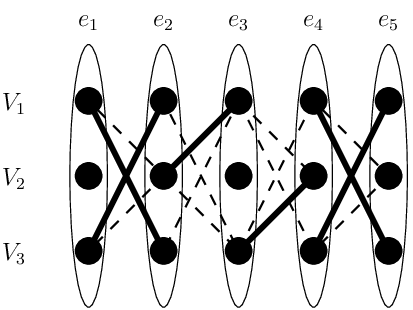}}
\caption{Diagrams for Claim~\ref{clm:forbiddencolouring1}}
\label{fig:clm:e_S(2,1,1)}
\end{figure}
\end{proof}

Recall that $P=e_1 \dots e_6$ is an oriented path of length~$5$ and not $4$-extensible.
Without loss of generality, we may assume that $\overleftarrow{e_3e_4}$.
By Claim~\ref{clm:forbiddencolouring1}(b) on the subpaths $e_1e_2e_3e_4$ and $e_4e_3e_2e_1$, we deduce that $\overleftarrow{e_1e_2}$.
Hence, $ \overrightarrow{e_2e_3}$ by Claim~\ref{clm:forbiddencolouring1}(a) on the subpath $e_1e_2e_3e_4$.
By Claim~\ref{clm:forbiddencolouring1}(b) taking $u_1u_2u_3u_4 = e_5e_4e_3e_2$, we have $ \overrightarrow{e_4e_5}$.
Therefore, the path $e_2 e_3 e_4 e_5 e_6$ satisfies (c) or~(d), so $P$ is $4$-extensible.

(iii)
Let $e_1, \dots , e_5$ be edges such that the edge-coloured path $e_1 \dots e_5$ in $G_S$ is isomorphic to $P$.
We will only consider the case when $a_2 = 1$ and $a'_1 = 0 = a'_2$ (since the arguments for the other cases are similar).
Hence, $e_2e_3$ is coloured $(0,1,2)$ in $G_M$ and so either $v_{1,2}v_{3,3} \in L_S(e_2,e_3)$ or $v_{1,3}v_{3,2} \in L_S(e_2,e_3)$.
If $v_{1,2}v_{3,3} \in L_S(e_2,e_3)$, then $\{v_{1,2}v_{3,3}, v_{2,2}v_{1,3}, v_{2,3}v_{1,4}, v_{2,4}v_{3,5}, v_{3,4} v_{2,5}\}$ (see Figure~\ref{fig:clm:(2,,)(,2,)intersection}~$(A)$) is a matching of size $5$ in $L_S(e_2,\dots, e_5)$.
If $v_{1,3}v_{3,2} \in L_S(e_2,e_3)$, then $\{v_{1,1}v_{2,2}, v_{1,2}v_{2,1}, v_{3,2}v_{1,3}, v_{2,3}v_{1,4}, v_{2,4}v_{3,5}, v_{3,4}v_{2,5}\}$ (see Figure~\ref{fig:clm:(2,,)(,2,)intersection}~$(B)$) is a matching of size $6$ in $L_S(e_1,\dots, e_5)$.
By Fact~\ref{fact:extensible}, both cases imply that $P$ is $3$-extensible.
\begin{figure}[tbp]
\centering
\subfloat[]{
\includegraphics[scale=0.6]{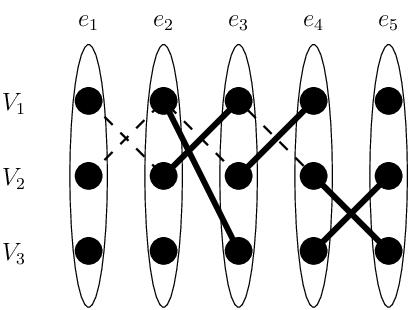}}
\subfloat[]{
\includegraphics[scale=0.6]{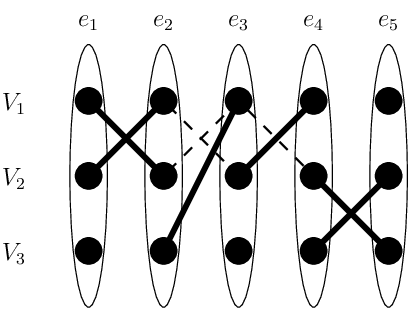}}
\caption{Diagrams for Lemma~\ref{lma:forbiddencolouring}(iii)}
\label{fig:clm:(2,,)(,2,)intersection}
\end{figure}

(iv)
Let $e_1, \dots , e_6$ be edges such that the edge-coloured path $e_1 \dots e_6$ in $G_S$ is isomorphic to $P$.
Since $e_2e_3$ is coloured $(0,1,2)$ in $G_M$, either $v_{3,2}v_{1,3} \in L_S(e_2,e_3)$ or $v_{1,2}v_{3,3} \in L_S(e_2,e_3)$.
If $v_{3,2}v_{1,3} \in L_S(e_2,e_3)$, then $\{v_{1,1}v_{2,2}, v_{2,1}v_{1,2}, v_{3,2}v_{1,3} \}$ is a matching of size~$3$ in $L_S(e_1, e_2, e_3)$.
If $v_{1,2}v_{3,3} \in L_S(e_2,e_3)$, then $\{ v_{2,2}v_{1,3}, v_{1,2}v_{3,3} \}$ is a matching of size $2$ in $L_S(e_2, e_3)$.
In summary, there exists a matching $M$ of size $i+1$ in $L_S(e_{3-i}, \dots , e_3)$ that is vertex-disjoint from $v_{2,3}$ for some $i \in \{1,2\}$.
By a similar argument, there exists a matching $M'$ of size $i'+1$ in $L_S(e_{4}, \dots , e_{4+i'})$ that is vertex-disjoint from $v_{1,4}$ for some $i' \in \{1,2\}$.
Therefore, $M \cup M' \cup \{v_{2,3}v_{1,4}\}$ is a matching of size $i+i'+3$ in $L_S( e_{3-i}, \dots, e_{4+i'} )$.
Moreover, Fact~\ref{fact:extensible} implies that $P$ is $5$-extensible.
\end{proof}


Recall that our aim is to find a large matching in $3$-partite $3$-graphs $H$.
For the rest of this section, we assume that $M$ is a matching in $H$ of maximal size $m$.
Let $X_i = V_i \setminus V(M)$ for $i \in [3]$ and $x = |X_1|$. 
Set $X= X_1 \times X_2 \times X_3$ and so $|X| = x^3$. 
Note that if $S \in X$, then $S$ is a triple $(x_1,x_2,x_3)$ with $x_i \in X_i$.
Let $G_S(M)$ be the matching graph defined above. 
Given $S \in X$ and $(a_1,a_2,a_3)$ with $0 \le a_1, a_2, a_3 \le 2$, define $G_S(a_1,a_2,a_3)$ to be the subgraph of $G_S$ induced by edges of colour $(a_1,a_2,a_3)$ and let $e_S(a_1,a_2,a_3) = |E(G_S(a_1,a_2,a_3))|$.

Given $\gamma >0$, define $G_S'(a_1,a_2,a_3)$ to be a subgraph of $G_S(a_1,a_2,a_3)$ such that the minimum degree $\delta(G_S'(a_1,a_2,a_3)) \ge \gamma  m/16$, if it exists.
Let $e_S'(a_1,a_2,a_3) = |E(G'_S(a_1,a_2,a_3))|$.
We further assume that $G_S'(a_1,a_2,a_3)$ is chosen such that $e_S'(a_1,a_2,a_3)$ is maximal.
Note that $G_S'(a_1,a_2,a_3)$ is a function of $\gamma$ and the value of $\gamma$ will always be known from the context. 
We need the following simple fact for graphs.
It is easily proved by consecutively removing a vertex $v$ with degree less than $\varepsilon n$.
Thus, its proof is omitted.
\begin{prp} \label{prp:subgraph}
Let $G$ be a graph of order $n$.
If $e(G) > 2 \varepsilon \binom{n}2$, then there exists an induced subgraph $G' \subseteq G$ with $\delta(G') \ge \varepsilon n$ and $e(G - G') < 2 \varepsilon \binom{n}2$.
\end{prp}

Proposition~\ref{prp:subgraph} implies that
\begin{align}
e_S(a_1,a_2,a_3) -  e_S'(a_1,a_2,a_3) \le \frac{\gamma}8  \binom{m}2 . \label{eqn:G'}
\end{align}
Let $G_S'(\cdot,\cdot,0)$ be the union of $G_S'(2,2,0)$, $G_S'(2,1,0)$ and $ G_S'(1,2,0)$ and let $e_S'(\cdot,\cdot,0) = |E(G_S'(\cdot,\cdot,0))|$.
Hence, for all $S \in X$, 
\begin{align}
e_S'(\cdot,\cdot,0) \ge e_S(2,2,0) + e_S(2,1,0)+ e_S(1,2,0) - \frac{3\gamma}8\binom{m}2. \label{eqn:e_S'(,,0)}
\end{align}
Define $G_S'(\cdot,0,\cdot)$, $G_S'(0,\cdot,\cdot)$, $e_S'(\cdot,0,\cdot)$ and $e_S'(0,\cdot,\cdot)$ similarly.

As mentioned at the beginning of this section, we would like to find a large matching of size at least $(1- \rho)n$ in $H$. 
Suppose $M$ is a matching in $H$ of maximum size with $m=|M|  < (1-\rho) n $.
Define $X_1,X_2,X_3$ and $X$ as above. 
Assume that we are in the ideal case that given $e,e' \in M$, $L_S(e,e') = L_{S'}(e,e')$ for $S, S' \in X$.
Thus, the matching graph $G_S$ is the same for all $S \in X$.
Since $\delta_1(H)$ is large, we are able to deduce some information on the distribution of colours in~$G_S$.
For example, $G_S$ is not monochromatic in colour $(0,0,0)$.
Most importantly, there exists an edge-coloured path $P$ in $G_S$ isomorphic to one of those listed in Lemma~\ref{lma:forbiddencolouring}.
Since $P$ is $p$-extensible for $p \le 5 \le |X_i|$, we can enlarge $M$ contradicting the maximality of~$M$.

In the next lemma, we study the colour distribution of $G_S$ for all $S \in X$.
We would also like to point out that there is no assumption on $\delta_1(H)$ in the hypothesis of the lemma.

\begin{lma} \label{lma:app1}
Let $0 <  \rho ,  \gamma < 10^{-6}$ and let $n$ be a sufficiently large integer.
Suppose that $H$ is a $3$-partite $3$-graph with partition classes $V_1,V_2,V_3$ each of size~$n$.
Let $M$ be a matching in $H$ of maximal size $m = |M| = (1- \rho)n$. 
Set $X_i = V_i \setminus V(M)$ for $i \in [3]$ and $X= X_1 \times X_2 \times X_3$.
Let $x = \rho n = |X_1| = |X_2| = |X_3|$.
Then the following statements hold:
\begin{itemize}
	\item[(i)] For all $S \in X$ and all $a_1 + a_2 + a_3 \ge 5$, $e_S(a_1,a_2,a_3) = 0$.
	\item[(ii)] For all but at most $3x^3/8$ sets $S \in X$, we have 
\begin{align}
e_S(2,1,1), e_S(1,2,1), e_S(1,1,2) < \frac{\gamma}2  \binom{m}2. \label{eqn:e(1,2,1)small}
\end{align}
	\item[(iii)] There are less than $x^3/24$ sets $S \in X$ such that 
\begin{align*}
|V(G_S'(0,\cdot,\cdot)) \cap V(G_S'(\cdot,0, \cdot))| \ge  \gamma m /8.
\end{align*}
	Moreover, similar statements hold for $|V(G_S'(0,\cdot,\cdot)) \cap V(G_S'(\cdot, \cdot,0))|$ and $ |V(G_S'(\cdot, 0 , \cdot)) \cap V(G_S'(\cdot, \cdot,0))|$. 
	\item[(iv)] There are less than $x^3/8$ sets $S \in X$ such that each of $V(G_S'(0,\cdot,\cdot))$, $V(G_S'(\cdot,0,\cdot))$ and $V(G_S'(\cdot,\cdot,0))$ has size at least $\left(1/3 +{\gamma}/8 \right)m$.
	\item[(v)] Let $S =(x_1,x_2,x_3)\in X$ such that $S$ satisfies \eqref{eqn:e(1,2,1)small} and $e(L_{x_i}[M]) \ge (1+ \alpha ) \binom{m}2$ for $i \in [3]$ with $\alpha >0$.
Then 
\begin{align*}
e_S'(\cdot,\cdot,0), e_S'(\cdot,0,\cdot), e_S'(0,\cdot,\cdot) \ge \left( \alpha - {7\gamma}/{8}  \right) \binom{m}2.
\end{align*}
	\item[(vi)] There are less than $x^3/8$ sets $S = (x_1,x_2,x_3) \in X$ such that $S$ satisfies \eqref{eqn:e(1,2,1)small} and $e(L_{x_i}[M] ) \ge (10/9 + \gamma ) \binom{m}2$ for all $i \in [3]$.
\end{itemize}
\end{lma}

\begin{proof}
Note that an edge $(e_1,e_2)$ with colour $(a_1,a_2,a_3)$ with $a_1 + a_2 + a_3 \ge 5$ is $1$-extensible by Lemma~\ref{lma:forbiddencolouring}(i).
If such an edge exists in $G_S(M)$ for some $S \in X$, then Definition~\ref{dfn:ext} implies that there exists a matching $M'$ of size~$3$ on $V(e_1) \cup V(e_2) \cup V(S)$.
Hence, we can enlarge $M$ by replacing $\{e_1,e_2 \}$ with $M'$.
This contradicts the maximality of $M$ and so (i) holds.

Suppose that (ii) is false, so without loss of generality there exists a subset $X' \subseteq X$ of size at least $x^3/8$ such that $e_S(1,2,1) \ge \frac{\gamma}2  \binom{m}2$ for every $S \in X'$.
Pick $S \in X'$.
By Proposition~\ref{prp:subgraph}, the subgraph $G'_S(1,2,1)$ exists with $|V(G'_S(1,2,1))| \ge \gamma m/16$.
Thus, for each $S \in X'$ there are at least $(\gamma m /32)^6$ copies of a path of length $5$ in~$G' \subseteq G_S(1,2,1)$.
Then by an averaging argument, there exists a copy $P_0 = e_1e_2e_3e_4e_5e_6$ of a path of length $5$ such that $P_0 \subseteq G_S(1,2,1)$ for at least $2^{-33} \gamma^6 x^3$ sets $S \in X'$.
Note that if $(e'_1,e'_2)$ is coloured $(1,2,1)$ in $G_S$, then there are $4$ possible structures for $L_S(e'_1,e'_2)$.
Thus, there are at least $2^{-43} \gamma^6 x^3 > 9x^2$ sets $S \in X$ such that $L_S(e_1,e_2,e_3,e_4,e_5,e_6)$ is identical, where we recall that $x = \rho n $ and $n$ is sufficiently large.
Moreover, there exist four such sets $S_1, \dots, S_4 \in X$ that are vertex-disjoint. 
Since a monochromatic path of length 5 with colour $(1,2,1)$ is $4$-extensible by Lemma~\ref{lma:forbiddencolouring}(ii), there exists a matching $M'$ of size~$7$ in $H[\bigcup_{i \in [6]} V(e_i) \cup \bigcup_{j \in [4] } V(S_j)]$ by Definition~\ref{dfn:ext}.
Hence, $(M - \{ e_1, \dots, e_6\} )\cup M' $ is a matching of size $m+1$, contradicting the maximality of $M$.
Thus (ii) holds.

To prove (iii), suppose that there are at least $x^3/24$ sets $S \in X$ such that $|V(G_S'(0,\cdot,\cdot)) \cap V(G_S'(\cdot,0, \cdot))| \ge  \gamma m /8$.
We will only consider the cases when there are at least $x^3/216$ sets $S \in X$ such that one of the following holds:
\begin{itemize}
\item[(a)] $|V(G_S'(0,2,1)) \cap V(G_S'(2,0,1))| \ge  \gamma m /72$ for all such $S$,
\item[(b)] $|V(G_S'(0,1,2)) \cap V(G_S'(1,0,2))| \ge  \gamma m /72$ for all such $S$.
\end{itemize}
The arguments for the other cases are similar.

First suppose that (a) holds, that is, there are at least $x^3/216$ sets $S \in X$ with $|V(G_S'(0,2,1)) \cap V(G_S'(2,0,1))| \ge \gamma m/72$.
Recall that $\delta(G_S'(0,2,1))$, $\delta( G_S'(2,0,1)) \ge \gamma m/16$.
Thus, for each such $S$ there are at least 
$$ \frac{\gamma  m}{72} 
\frac{\gamma  m}{16}  \left(\frac{\gamma  m}{16}  -1\right) 
\left(\frac{\gamma  m}{16}  -2\right) 
\left(\frac{\gamma  m}{16}  -3\right) 
\ge \frac19 \left(\frac{\gamma m }{16}\right)^5$$
copies of 4-paths $P_0 = e_1 \dots e_5$ in $G_S$ such that $e_i e_{i+1}$ has colour $(0,2,1)$ for $i \in [3]$ and $e_4 e_{5}$ has coloured $(2,0,1)$.
By a similar averaging argument used in the proof of~(ii), there exist edges $e_1,e_2,\dots, e_5 \in M$ and vertex-disjoint $S_1, S_2,S_3 \in X$ such that, for each $i \in [3]$, $e_1 \dots e_5$ is isomorphic to the 4-path $P_0$ in $G_{S_i}$ and $L_{S_i}(e_1,\dots, e_{5})$ is identical.
Since $P_0$ is $3$-extensible by Lemma~\ref{lma:forbiddencolouring}(iii), there exists a matching $M'$ of size~$6$ in $H[\bigcup_{i \in [5]} V(e_i) \cup \bigcup_{j \in [3] } V(S_j)]$ by Definition~\ref{dfn:ext}.
Hence, $(M -\{ e_1, \dots, e_5\}) \cup M'$ is a matching of size $m+1$, contradicting the maximality of $M$.

Next, suppose that (b) holds.
Let $P_0' = v_1 \dots v_6$ be an edge-coloured path of length $5$ such that $v_i v_{i+1}$ is coloured $(0,2,1)$ for $i \in [3]$ and $(2,0,1)$ for $i \in \{4,5\}$.
By a similar averaging argument used above, there exist edges $e_1,e_2,\dots, e_6 \in M$ and vertex-disjoint $S_1, \dots ,S_5 \in X$ such that, for each $i \in [5]$, $e_1 \dots e_6$ is isomorphic to the path $P_0'$ in $G_{S_i}$ and $L_{S_i}(e_1,\dots, e_{6})$ is identical.
Since $P_0'$ is $5$-extensible by Lemma~\ref{lma:forbiddencolouring}(iv), there exists a matching $M'$ of size~$7$ in $H[\bigcup_{i \in [6]} V(e_i) \cup \bigcup_{j \in [5] } V(S_j)]$ by Definition~\ref{dfn:ext}.
Hence, $(M -\{ e_1, \dots, e_6\}) \cup M'$ is a matching of size $m+1$, contradicting the maximality of $M$.
Thus (iii) holds, which implies (iv).

To prove (v), define $G_S(a_1,\ast,\ast)$ to be the subgraph $\bigcup_{0 \le b_2,b_3 \le 2} G_S(a_1,b_2,b_3)$ and write $e_S(a_1,\ast,\ast)$ for $|E(G_S(a_1,\ast,\ast))|$.
Similarly, define $G_S(\ast,a_2,\ast)$, $G_S(\ast,\ast,a_3)$, $e_S(\ast,a_2,\ast)$ and $e_S(\ast,\ast,a_3)$.
Recall that for a vertex $x$, $L_{x}[M] = \bigcup_{\{e_i,e_j\} \in \binom{M}{2}} L_{x}[V(e_i), V(e_{j})] $.
So
\begin{align}
	e(L_{x_1}[M])
& =  2 e_S(2, \ast, \ast) + e_S(1, \ast, \ast) 
=e_S(2, \ast, \ast) + \binom{m}2 - e_S(0, \ast, \ast)\nonumber  \\
	& \le e_S(2, \ast, \ast) + \binom{m}2 -e_S'(0,\cdot,\cdot)- e_S(0,2,0). \nonumber 
\end{align}
By our assumption on $e(L_{x_1}[M])$, the inequality above implies that 
\begin{align}
	e_S(2, \ast, \ast)  \ge & \alpha \binom{m}2 + e_S'(0,\cdot,\cdot) + e_S(0,2,0)\label{eqn:L_x(M)2}.
\end{align}
On the other hand, by \eqref{eqn:e(1,2,1)small} and \eqref{eqn:e_S'(,,0)} we have 
\begin{align}
	e_S(2, \ast, \ast)  & \le e_S'(\cdot,\cdot,0) + e_S'(\cdot,0,\cdot) + e_S(2,1,1)+e_S(2,0,0) +  \frac{3\gamma}{4}  \binom{m}2 \nonumber \\
	& \le  e_S'(\cdot,\cdot,0) + e_S'(\cdot,0,\cdot) + e_S(2,0,0) + \frac{7\gamma}{8}  \binom{m}2 . 
	\nonumber
\end{align} 
Thus, together with \eqref{eqn:L_x(M)2}, we have 
\begin{align}
  e_S'(\cdot,\cdot,0) + e_S'(\cdot,0,\cdot) + e_S(2,0,0)  \ge \left( \alpha - \frac{7\gamma}{8}  \right) \binom{m}2 + e_S'(0,\cdot,\cdot) + e_S(0,2,0),
\nonumber
\end{align}
and by a similar argument (i.e. swapping the indices) we have
\begin{align}
 e_S'(\cdot,\cdot,0) + e_S'(0,\cdot,\cdot) + e_S(0,2,0) \ge \left( \alpha - \frac{7\gamma}{8}   \right) \binom{m}2 + e_S'(\cdot,0,\cdot) + e_S(2,0,0) .
\nonumber
\end{align}
Therefore, by combining the last two inequalities, we obtain
\begin{align}
e_S'(\cdot,\cdot,0) \ge \left( \alpha - {7\gamma}/{8}  \right) \binom{m}2 . \nonumber
\end{align}
We obtain similar statements for $e_S'(\cdot,0,\cdot)$ and $e_S'(0,\cdot,\cdot)$.
Hence, (v) holds.

Finally suppose that (vi) is false.
Thus there are at least $x^3/8$ sets $S = (x_1,x_2,x_3) \in X$ such that $S$ satisfies \eqref{eqn:e(1,2,1)small} and $e(L_{x_i}[M]) \ge (10/9 + \gamma ) \binom{m}2$ for all $i \in [3]$.
For each such $S$, (v) implies that $e_S'(\cdot,\cdot,0) \ge \left( 1/9 + {\gamma}/{8}  \right) \binom{m}2$.
Hence, $G'_S(\cdot,\cdot,0)$ spans at least $(1/3 +\gamma/8)m$ vertices in $G_S(M)$, and similar statements hold for 
$G'_S(\cdot,0, \cdot)$ and $G'_S(0,\cdot,\cdot)$.
However this contradicts~(iv).
\end{proof}

Here, we prove Theorem~\ref{thm:asymptotic} implying that $m'_1(3,n) \sim 5n^2 /9$ for large~$n$.

\begin{proof}[Proof of Theorem~\ref{thm:asymptotic}]
Fix $0 < \gamma <10^{-6}$ and assume that $n_0$ is sufficiently large..
Let $H$ be a $3$-partite $3$-graph with $n> n_0$ vertices in each class and $\delta_1(H) \ge \left( 5/9+ \gamma  \right) n^2$.
Note that $\gamma$ satisfies the hypothesis of Lemma~\ref{lma:absorptionlemma} with $k = 3$ and $l=1$.
Let $M$ be the matching given by Lemma~\ref{lma:absorptionlemma} and so $|M| \le 2 \gamma^3 n$.
Let $H' = H \setminus V(M)$ be the $3$-partite $3$-graph with partition classes $V'_1,V'_2,V'_3$ each of size $n'$, where $V'_i = V_i \setminus V(M)$ and  $n' = n - |M|$.
Note that 
\begin{align}
	\delta_1(H') \ge \left( 5/9 + \gamma  \right) n^2 - 4 \gamma^3 n^2 
\ge \left( 5/9 + \gamma/2  \right) n'^2. \nonumber
\end{align}
Let $M_1$ be a largest matching in $H'$ of size~$m$.
Next we claim that $m \ge (1-6 \gamma^6)n'$.
Suppose the contrary, so $m < (1-6 \gamma^6)n'$.
Let $X_i = V'_i \setminus V(M_1)$ for $i \in [3]$ and $X = X_1 \times X_2 \times X_3$.
Without loss of generality, we may suppose that $x = |X_1| = |X_2| = |X_3| =  \rho n$ (by assigning additional balanced $3$-sets to $M$ and omiting floors and ceilings for clarity of presentation).
Given vertex $x \notin V(M)$, the number of $2$-sets $T \subseteq V$ such that $T \notin L_x[M]$, is at most $2x m + m \le 2 \rho  n'^2 \le \gamma n'^2/2$.
Hence, by the assumption on $\delta_1(H')$, for every $x \in X$,
\begin{align}
	e(L_x[M] )\ge  \deg_{H'} (x) - {\gamma  n'^2}/2 
\ge \left( 5/9 + {\gamma }/2 \right)n'^2 
\ge \left( {10/}9 + \gamma  \right) \binom{m}2 \nonumber.
\end{align}
However, this implies a contradiction by  Lemma~\ref{lma:app1}(ii) and~(iv).
Therefore, $m \ge (1-6 \gamma^6)n'$ as claimed.

Let $W = V(H') \setminus V(M_1)$.
Note that $W$ is balanced with at most $6 \gamma^6 n' < 6 \gamma^6 n$ vertices.
Thus, there exists a matching $M_2$ covering exactly the vertices of $V(M) \cup W$ by the property of $M$, and so $M_1 \cup M_2$ is a perfect matching in~$H$.
\end{proof}


\subsection{Proof of Lemma~\ref{lma:exact}} \label{sec:epsilon-close}
Lemma~\ref{lma:app1} provides us with some information on the matching graph $G_S(M)$ when the matching $M$ is not large enough. 
To prove Lemma~\ref{lma:exact}, we analyze the structure of $G_S(M)$ further.
The proof involves a series of claims, which gives the structure of $H$.

\begin{proof}[Proof of Lemma~\ref{lma:exact}]
Without loss of generality, we may assume that $n_0$ is sufficiently large.
Let $V_1$, $V_2$ and $V_3$ be the partition classes of $H$ and let $M$ be a matching in $H$ with maximal size.
Let $m = |M|$, so $n - m \ge  \rho n$.
By assigning additional balanced $3$-sets to $M$ and omitting floors and ceilings for clarity of presentation, we may assume that $m = (1- \rho)n$.
Set $Y_i = V(M) \cap V_i$ and $Y = V(M)$.
Let $X_i = V_i \setminus Y_i$ for $i \in [3] $ and $x = |X_i| =  \rho n$.
For every $S \in X_1 \times X_2 \times X_3$, define $G_S$, $G_S'(a_1,a_2,a_3)$ and $G_S'(\cdot,\cdot,0)$ as in Section~\ref{sec:matchinggraphs}.
We abuse the notation by letting $X$ to mean both $X_1 \times X_2 \times X_3$ and $X_1 \cup X_2 \cup X_3$, but it will be clear from the context.

In the next claim, we bound the number of edges of type $XXY$ in~$H$, (where $X = X_1 \cup X_2 \cup X_3$).
Recall that for a vertex $v \in V(H)$ and disjoint vertex sets $U,U' \subseteq V(H)$, $L_v[U,U']$ is the bipartite subgraph of $L_v$ induced by the vertex classes $U$ and~$U'$.

\begin{clm} \label{clm: e(L_S[X,M])}
For all $i \in [3]$, all but at most $x/8$ vertices $x_i \in X_i$ satisfy $e(L_{x_i}[X,Y]) \le  (1 + \sqrt{\gamma})mx$.
\end{clm}

\begin{proof}
Suppose the claim is false for $i=1$ say.
Pick $x_1 \in X_1$ such that $e(L_{x_1}[X,Y])   > (1 + \sqrt{\gamma})mx$.
For an edge $e \in M$, we say that $e$ is \emph{good for} $x_1$ if each of $v_2 = V(e) \cap V_2$ and $v_3 = V(e) \cap V_3$ has degree at least $1$ in $L_{x_1}[V(e),X]$.
We claim that there are at least $2\sqrt{\gamma}m/3$ good edges for $x_1$.
Indeed this is true or else we have
\begin{align}
	e(L_{x_1}[X,Y] ) & <  x \cdot (1- 2\sqrt{\gamma}/3) m  + 2 x \cdot 2 \sqrt{\gamma} m/3 
 =  (1 + 2\sqrt{\gamma}/3) m x,\nonumber
\end{align}
a contradiction.
Since there are at least $x/8$ such $x_1 \in X_1$, by an averaging argument there exists an edge $e \in M$ that is good for distinct $x_1,x'_1\in X_1$.
Thus there exist $u \in X_{3}$ and $u' \in X_{2}$ such that $x_1 v_{2} u$ and $x_1' u' v_{3}$ are vertex-disjoint edges in~$H$, where $v_i = V(e) \cap V_i$.
Hence, we can enlarge $M$ (by replacing $e$ with $x_1 v_{2} u$ and $x_1' u' v_{3}$) contradicting the maximality of~$M$.
\end{proof}

Denote by $X'$ the set of $S = (x_1,x_2,x_3) \in X_1 \times X_2 \times X_3$ such that 
\begin{itemize}
	\item[(a)] $e_S(2,1,1), e_S(1,2,1), e_S(1,1,2) \le \frac{\gamma}2 \binom{m}2 $,
	\item[(b)] $e(L_{x_i}[X,Y]) \le  (1 + \sqrt{\gamma})mx$ for all $i \in[3]$,
	\item[(c)] $|V(G_S'(0,\cdot,\cdot)) \cap V(G_S'(\cdot,0, \cdot))| \le \gamma m /8$, and similar arguments hold, where we swap the indices.
\end{itemize}
Note that $|X'| \ge x^3/8$ by Lemma~\ref{lma:app1}(ii), (iii) and Claim~\ref{clm: e(L_S[X,M])}.
The following claim shows that $m = |M|  \ge (1-3\sqrt{\gamma/2})n$.
Recall that $L_{x}[M] = \bigcup_{\{e_i,e_j\} \in \binom{M}{2}} L_{x}[V(e_i), V(e_{j})] $ for $x \in X$.

\begin{clm} \label{clm:L_x(M)}
For each $S =(x_1,x_2,x_3) \in X'$ and all $i \in [3]$, $e(L_{x_i}[M])  \ge  \left({10}/9 - 2\gamma \right) \binom{m}2$.
Moreover, $m \ge (1-3\sqrt{\gamma/2})n$.
\end{clm}

\begin{proof}
Pick $i \in[3]$.
Note that there are at most $m$ sets $T \in \binom{V}{2}$ such that $\{x_i\} \cup T$ is an edge and $T$ is contained in an edge $e \in M$.
By the maximality of $M$, every edge containing $x_i$ must intersect~$M$.
Recall that $n = m+x$ and $\sqrt{\gamma}x = \sqrt{\gamma} \rho n \ge 1$ as $n$ is large, so
\begin{align}
	e(L_{x_i}[M]) & \ge  \deg_{H}(x_i)- e(L_{x_i}[X,Y]) - m \nonumber \\
	& \ge   (5/9 - \gamma)n^2 -(1+ 2\sqrt{\gamma} ) m x \nonumber \\
	& = 	(5/9 - \gamma)(m^2+x^2) + (1/9 - 2 \gamma - 2\sqrt{\gamma})mx \nonumber \\
	& \ge (5/9 - \gamma)(m^2+x^2) \label{eqn:L_x(M)3:exact}.
\end{align}
Hence $e(L_{x_i}[M])  \ge  \left({10}/9 - 2\gamma \right) \binom{m}2$.
Now suppose that $m < (1-3\sqrt{\gamma/2})n$ and so $x > 3\sqrt{\gamma/2} n \ge 3\sqrt{\gamma/2} m$.
Hence, \eqref{eqn:L_x(M)3:exact} becomes
\begin{align}
	e(L_{x_i}[M]) & \ge  (5/9 - \gamma)\left(1+ {9\gamma}/2 \right)m^2 
	\ge (10/9 + \gamma) \binom{m}2 \nonumber
\end{align}
for all $i \in [3]$ and $S =(x_1,x_2,x_3)\in X'$.
Since (a) implies \eqref{eqn:e(1,2,1)small}, this contradicts Lemma~\ref{lma:app1}(vi) as $|X'| \ge x^3/8$.
This completes the proof of the claim.
\end{proof}

In the next claim, we show that for each $S \in X'$ almost all edges in $G_S$ have colours $(2,2,0)$, $(2,0,2)$, $(0,2,2)$ or $(1,1,1)$.

\begin{clm} \label{clm:edgedistribution}
For each $S \in X'$, $e_S'(2,2,0), e_S'(2,0,2), e_S'(0,2,2) \ge  \left(1/9 - 23\gamma  \right)\binom{m}2$ and $e_S(1,1,1) \ge  \left(2/3 - 55\gamma  \right)\binom{m}2$.
In particular, there are at most $125 \gamma \binom{m}2$ edges in $G_S(M)$ not coloured by $(2,2,0)$, $(2,0,2)$, $(0,2,2)$ or $(1,1,1)$.
\end{clm}

\begin{proof}
By Claim~\ref{clm:L_x(M)} and Lemma~\ref{lma:app1}(v) taking $\alpha = 1/9-2 \gamma$, we have 
\begin{align*}
	e'_S(\cdot,\cdot,0), e'_S(\cdot,0,\cdot), e'_S(0,\cdot,\cdot) \ge  \left( 1/9 - 3 \gamma  \right) \binom{m}2.
\end{align*}
Hence, $|V(G_S'(\cdot,\cdot,0))|, |V(G_S'(\cdot,0,\cdot))|, |V(G_S'(0,\cdot,\cdot))| \ge (1/3 - 5 \gamma) m $. 
Together with~(c) we can bound $|V(G_S'(\cdot,\cdot,0))|$, $|V(G_S'(\cdot,0,\cdot))|$, $|V(G_S'(0,\cdot,\cdot))|$ from above by $(1/3 + 11 \gamma) m$.
Moreover,
\begin{align}
 e_S'(2,2,0) \le 	e_S'(\cdot,\cdot,0) \le \left(\frac{1}9 + \frac{15 \gamma}{2} \right) \binom{m}2. \label{eqn:boundone(,,2)}
\end{align}
Denote by $q_S(j)$ the number of edges $(e_1,e_2)$ in $G_S$ such that $(e_1,e_2)$ is coloured $(a_1,a_2,a_3)$ with $a_1+a_2+a_3 = j$.
Clearly, $\sum_{j}q_S(j) =  \binom{m}2$.
Lemma~\ref{lma:app1}(i) implies that $q_S(j) = 0$ for $j \ge 5$.
In addition, by property (a) and \eqref{eqn:G'} we have 
\begin{align}
	q_S(4)& =  e_S(2,2,0) + e_S(2,0,2) +e_S(0,2,2) \nonumber \\&+ e_S(2,1,1) +e_S(1,2,1)+e_S(1,2,1) \nonumber \\
	 & \le   e_S'(2,2,0) + e_S'(2,0,2) +e_S'(0,2,2) + 2\gamma \binom{m}2. \nonumber
\end{align}
Therefore,
\begin{align}
 	e(L_S[M]) & =  \sum_{j \le 4} j q_S(j) \le 3\binom{m}2 + q_S(4) - \sum_{j \le 2}q_S(j) \nonumber \\
	& \le   e_S'(2,2,0) + e_S'(2,0,2) +e_S'(0,2,2) + \left(3+ 2 \gamma  \right) \binom{m}2- \sum_{j \le 2}q_S(j). \nonumber 
\end{align}
Recall that $L_S[M] = L_{x_1}[M] \cup L_{x_2}[M] \cup L_{x_3}[M]$, where $S = (x_1,x_2,x_3)$.
Hence, Claim~\ref{clm:L_x(M)} implies that 
\begin{align}
	e_S'(2,2,0) + e_S'(2,0,2) +e_S'(0,2,2)  \ge \left(\frac13 - 8 \gamma  \right)\binom{m}2 + \sum_{j \le 2}q_S(j).\label{eqn:6.41}
\end{align}
Together with~\eqref{eqn:boundone(,,2)}, we deduce that each of $e_S'(2,2,0)$, $e_S'(2,0,2)$ and $ e_S'(0,2,2)$ is at least $\left( 1/9- 23\gamma  \right)\binom{m}2$.
Moreover \eqref{eqn:6.41} becomes $ \sum_{j \le 2} q_S(j) \le  31 \gamma \binom{m}2$.
By \eqref{eqn:e_S'(,,0)} and \eqref{eqn:boundone(,,2)}, we have 
\begin{align*}
	\binom{m}2 & = e(G_S) =  \sum_{j \le 2} q_S(j) + q_S(3) + q_S(4) \\
	& \le 31 \gamma \binom{m}2 + e_S(1,1,1) + e_S'(\cdot,\cdot,0) + e_S'(\cdot,0,\cdot )+ e_S'(0,\cdot,\cdot ) + \frac{9\gamma}8 \binom{m}2 \\
	& \le  \left(1/3+ 55\gamma  \right)\binom{m}2 + e_S(1,1,1),
\end{align*}
where the last inequality is due to~\eqref{eqn:boundone(,,2)}.
Thus, $e_S(1,1,1) \ge  \left(2/3 - 55\gamma  \right)\binom{m}2$ as required.
\end{proof}

Next we show that there exist two vertex-disjoint sets $S_1, S_2 \in X'$ such that the matching graphs $G_{S_1}$ and $G_{S_2}$ are virtually the same.

\begin{clm} \label{clm:S_1S_2}
There exist two vertex-disjoint sets $S_1, S_2 \in X'$ such that $G_{S_1}$ and $G_{S_2}$ have at most $750 \gamma \binom{m}2$ edges that are coloured differently or not coloured by $(2,2,0)$, $(2,0,2)$, $(0,2,2)$ nor $(1,1,1)$.
\end{clm}

\begin{proof}
Recall that $|X'| \ge x^3/8$.
By~\cite{MR0183654} (taking $X'$ to be the edge set on the vertex set $X_1 \cup X_2 \cup X_3$), we know that there exist distinct $x_i,x'_i \in X_i$ for each $i\in[3]$ such that $ \{x_1,x'_1\} \times \{x_2,x'_2\} \times \{x_3,x'_3\} \subseteq X'$.
Set $S_1 = (x_1,x_2,x_3)$ and $S_2 = (x'_1,x'_2,x'_3)$.
To prove the claim, it is sufficient to show that if $S, S' \in X'$ differ only in one vertex, then there are at most $250 \gamma \binom{m}2$ edges that are coloured differently in $G_{S}$ and $G_{S'}$ or not coloured by $(2,2,0)$, $(2,0,2)$, $(0,2,2)$ nor $(1,1,1)$.
Without loss of generality, we may assume that $S=(x_1,x_2,x_3)$ and $S'=(x_1',x_2,x_3)$.
Note that for distinct $e_1,e_2 \in M$, if $(e_1,e_2)$ is of type $(a_1,a_2,a_3)$ with respect to $S$, then $(e_1,e_2)$ is of type $(a_1',a_2,a_3)$ with respect to $S'$.
Thus, if $(e_1,e_2)$ is coloured differently in $G_{S}$ and $G_{S'}$, then $(e_1,e_2)$ is not coloured by $(2,2,0)$, $(2,0,2)$, $(0,2,2)$, $(1,1,1)$ in $G_{S}$ or $G_{S'}$.
By Claim~\ref{clm:edgedistribution}, all but at most $125 \gamma  \binom{m}2$ edges are coloured by $(2,2,0)$, $(2,0,2)$, $(0,2,2)$, $(1,1,1)$ for $G_{S}$ and similarly for $G_{S'}$.
So the claim follows.
\end{proof}

Fix $S_1 = (x_1,x_2,x_3)$ and $S_2 =(x_1',x_2',x_3')$ in $ X'$ satisfying Claim~\ref{clm:S_1S_2} for the rest of the proof.
Let $G$ be the edge-coloured subgraph of a complete graph with vertex set $M$ induced by the edges that are coloured the same in $G_{S_1}$ and $G_{S_2}$ with colours in $\{(1,1,1)$, $(2,2,0)$, $(2,0,2)$, $(0,2,2)\}$.
By Claim~\ref{clm:S_1S_2}, $G$ has at least $(1-750\gamma)\binom{m}2$ edges.
By Proposition~\ref{prp:subgraph} and removing at most $8 \gamma \binom{m}2$ additional edges in $G$, we may further assume that the subgraph induced by edges of one colour (after removing the isolated vertices) has minimum degree at least $\gamma m$.
Call the resulting subgraph $G'$ and note that $e(G') \ge (1- 758 \gamma) \binom{m}2$.
Moreover, by Claim~\ref{clm:edgedistribution} we have
\begin{align}
	e(G'(2,2,0)), e(G'(2,0,2)), e(G'(0,2,2)) & \ge  \left(1/9- 790\gamma  \right)\binom{m}2, \label{eqn:e(G'(2,2,0)}
\end{align}	
where $G'(a_1,a_2,a_3)$ is the monochromatic subgraph of $G'$ induced by the edges of colour $(a_1,a_2,a_3)$ with all isolated vertices removed.

Define $M_1$ to be the set of edges $e \in M$ such that $e$ is an vertex in $V(G'(0,2,2))$.
Similarly define $M_2$ to be the set of edges $ e \in V(G'(2,0,2))$ and $M_3$ to be the set of edges $ e \in V(G'(2,2,0))$.
In the next claim, we study the induced subgraph of $G$ induced on the vertex set $M_1 \cup M_2 \cup M_3$.
Recall that $\{v_{i,j} \}= V_i \cap V(e_j)$ for $e_j \in M$.

\begin{clm} \label{clm:exactconfig}
\begin{itemize}
\item[(a)]$|M_1|$, $|M_2|$, $|M_3| \ge \left(1/3- 1200\gamma  \right) m$.
  \item[(b)] $M_1$, $M_2$ and $M_3$ are pairwise disjoint.
 \item[(c)] Let $e_1 \in M_{i}$ and $e_2 \in M_{i'}$ with $i \ne i'$.
Then each edge of $L_{S_1}(e_1,e_2) \cup L_{S_2}(e_1,e_2)$ is incident with $v_{i,1}$ or $v_{i',2}$.
	Moreover if $e_1e_2 \in E(G')$, then $e_1e_2$ is coloured $(1,1,1)$ and $L_{S_1}(e_1,e_2) = L_{S_2}(e_1,e_2)$.
	\item[(d)]Let $e_1,e_2 \in M_{1}$.
	If $e_1e_2 \in E(G')$, then $e_1e_2$ is coloured $(0,2,2)$.
Similar statement holds for $M_2$ and $M_3$.
	\item[(e)] If $e_1,e_2 \in M_1 \cup M_2 \cup M_3$ and $e_1e_2 \in E(G')$, then $L_{S_1}(e_1,e_2) = L_{S_2}(e_1,e_2)$.
\end{itemize}

\end{clm}

\begin{proof}
Let $S_1 = (x_1,x_2,x_3)$ and $S'_2 = (x'_1,x'_2,x'_3)$.
Since $\gamma < 10^{-6}$, (a) holds by~\eqref{eqn:e(G'(2,2,0)}.


To prove~(b), suppose that $M_1 \cap M_2 \ne \emptyset$.
By the minimum degree of $G'(0,2,2)$ and $G'(2,0,2)$, there exists a path $e_1e_2e_3e_4$ in~$G'$ such that $e_1e_2e_3$ is coloured $(0,2,2)$ and $e_3e_4$  is coloured $(2,0,2)$.
Therefore, there exists a matching 
\begin{align*}
M' = \{v_{1,1}v_{2,2}, v_{1,2} v_{2,1}, v_{3,2} v_{1,3} , v_{2,3} v_{3,4}, v_{3,3} v_{2,4}  \}
\end{align*}
of size 5 in $L_{S_i}(e_1, e_2,e_3, e_4)$ for $i = 1,2$, see Figure~\ref{fig:clm:exactconfig}(A).
This implies that 
\begin{align*}
M'' = \{x_3 v_{1,1}v_{2,2}, x'_3v_{1,2} v_{2,1}, x_2v_{3,2} v_{1,3} , x_1 v_{2,3} v_{3,4}, x'_1 v_{3,3} v_{2,4}  \}
\end{align*}
is a matching of size $5$ on $H[V(\{e_1, \dots, e_4\} \cup S_1 \cup S_2)]$.
Thus, $M \cup M'' - \{e_1, \dots, e_4\}$ is a matching of size $m+1$, contradicting the maximality of~$|M|$.
Therefore (b) follows.

To prove~(c), without loss of generality, we may assume that $i = 1$ and $i' =2$.
Let $e_0,e_3 \in M \setminus \{e_1,e_2\}$ such that $e_0 \ne e_3$ and $e_0e_1$ and $e_2e_3$ are of colours $(0,2,2)$ and $(2,0,2)$ respectively in $G'$.
Note that $e_0$ and $e_3$ exist by the minimum degree of $G'(0,2,2)$ and $G'(2,0,2)$.
Suppose (c) is false, so $v_{2,1} v_{3,2} \in  E(L_{S_1}(e_1,e_{2}))$ say.
Then there exists a matching 
\begin{align*}
M' = \{ v_{1,0} v_{3,1}, v_{3,0}v_{1,1}, v_{2,1} v_{3,2}, v_{1,2}v_{2,3}, v_{1,3} v_{2,2}\}
\end{align*}
of size $5 $ in $L_{S_1}(e_0,e_1,e_2,e_3)$, see Figure~\ref{fig:clm:exactconfig}(B).
Moreover, $M' -v_{2,1} v_{3,2}$ is a subgraph of $L_{S_2}(e_0,e_1,e_2,e_3)$.
Therefore
\begin{align*}
M'' = \{ x_2v_{1,0} v_{3,1}, x_2'v_{3,0}v_{1,1}, x_1v_{2,1} v_{3,2}, x_3 v_{1,2}v_{2,3}, x_3'v_{1,3} v_{2,2}\}
\end{align*}
is a matching of size $5$ on $H[V(\{e_0, \dots, e_3\} \cup S_1 \cup S_2)]$.
Thus, $M \cup M'' - \{e_0, \dots, e_3\}$ is a matching of size $m+1$, contradicting the maximality of~$M$.
Therefore (c) follows.

To prove (d), let $e_0,e_3 \in M_1 \setminus \{e_1,e_2\}$ such that $e_0 \ne e_3$ and $e_0e_1$ and $e_2e_3$ are both coloured $(0,2,2)$ in $G'$, which exist by the minimum degree of $G'(0,2,2)$.
Suppose that $e_1e_2$ is coloured by one of $\{(2,2,0)$, $(2,0,2)$, $(1,1,1)\}$.
Without loss of generality, we may assume that $v_{2,2}v_{3,1} \in E(L_{S_i}(e_1,e_2))$ for $i =1,2$.
Note that there exists a matching 
\begin{align*}
M' = \{v_{1,0} v_{2,1},  v_{2,0}v_{1,1}, v_{3,1}v_{2,2}, v_{1,2}v_{3,3}, v_{3,2}v_{1,3} \}
\end{align*}
of size $5$ in $L_{S_i}(e_0, e_1, e_2,e_3)$ for $i=1,2$, see Figure~\ref{fig:clm:exactconfig}(C).
By similar arguments used in (b) and (c), this implies a contradiction and so (d) holds.

Note that if $e_1e_2$ is coloured $(0,2,2)$ in $G'$ then $L_{S_1}(e_1,e_2) = L_{S_2}(e_1,e_2)$.
Therefore, in order to prove (e), it is enough to consider the case when $e_1e_2$ is coloured $(1,1,1)$ in $G'$.
Since $e_1,e_2 \in M_1 \cup M_2 \cup M_3$ and $c(e_1e_2) = (1,1,1)$, (d) implies that $e_1 \in M_{i}$ and $e_2 \in M_{i'}$ with $i \ne i'$.
Hence, $L_{S_1}(e_1,e_2) = L_{S_2}(e_1,e_2)$ by~(c).
\begin{figure}[tbp]
\centering
\subfloat[]{
\includegraphics[scale=0.6]{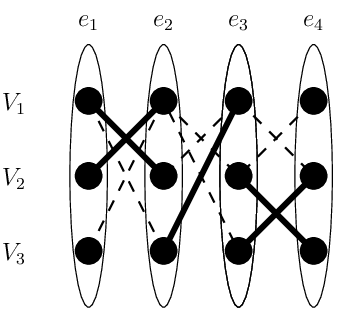}}
\subfloat[]{
\includegraphics[scale=0.6]{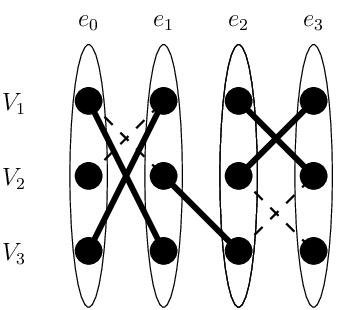}}
\subfloat[]{
\includegraphics[scale=0.6]{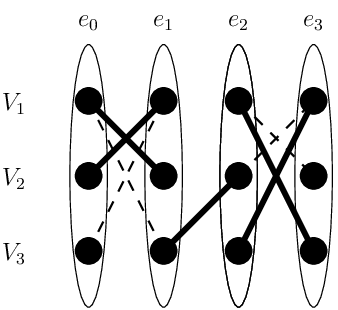}}
\caption{Diagrams for Claim~\ref{clm:exactconfig}}
\label{fig:clm:exactconfig}
\end{figure}
\end{proof}

Recall Claim~\ref{clm:L_x(M)} that $m \ge (1-3\sqrt{\gamma/2})n$.
For each $i \in[3]$, pick $M'_i \subseteq M_i$ of size exactly  $ \left(1/3- 1200\gamma  \right) m  \ge (1-\gamma'')n/3$.
Set $M' = \bigcup_{i \in [3]} M_i'$.
Now we define $H'$ to be the $3$-partite $3$-graph in $H$ induced by the vertex set $V(M')$.
Hence, each partition has size $n' = 3|M_i'| \ge (1 - \gamma'') n$.
Let $V'_i = V(H') \cap V_i$, $W_i = V(M'_i) \cap V_i'$ and $U_i = V'_i \setminus W_i$.
Set $V' = V(H')$, $W = \bigcup_{i \in [3]} W_i$ and $U= \bigcup_{i \in [3]} U_i$.
Note that $|W_i| = n'/3$ and $|U_i| = 2n'/3$.
Our aim is to show that $H'$ is $8 \gamma''$-close to $H'_3(n';n')$.
The next two claims show that the number of edges in $H'$ of type $UUU$ is small.

\begin{clm} \label{clm:Utriangle}
Let $\{e_1, e_2, e_3\}$ be a triangle in $G'$ with $e_{j} \in M'$.
Then $u_1u_2u_3$ is not an edge in $H$ for $u_j \in V(e_j) \cap U$ for $j \in [3]$.
\end{clm}

\begin{proof}
Suppose the contrary let $u_1u_2u_3$ be an edge in $H$ with $u_j \in V(e_j) \cap U$.
Without loss of generality, we may assume that $u_j = v_{j,j} = V(e_j) \cap  V'_{j}$. 
Further suppose that $e_j \in M_{i_j}$.
Since $V_{i_j} \cap V(e_j)  \in W_{i_j}$, we have $i_j \ne j$.
In order to obtain a contradiction, it is sufficient to show that there exists a matching of size 3 in $L_{S_1}(e_1,e_2) \cup L_{S_1}(e_2,e_3) \cup L_{S_1}(e_3,e_1)$ avoiding the vertices $u_1,u_2,u_3$, which then implies that we can enlarge $M$ together with $S_1$ and $S_2$.

First suppose that $i_1,i_2,i_3$ are distinct.
Without loss of generality we may assume that $e_1 \in M'_2$, $e_2 \in M'_3$ and $e_3 \in M'_1$.
Since $e_1e_2e_3$ forms a triangle in $G'$, by Claim~\ref{clm:exactconfig}(c) and~(e) we deduce that $(e_i,e_j)$ is coloured $(1,1,1)$ for $i \ne j$.
Moreover, we can determine the exact structure of $L_S(e_i,e_j)$ for $i \ne j$.
In particular, $\{v_{1,2} v_{2,1}$, $v_{1,3} v_{3,1}$, $v_{2,3} v_{3,2}\}$ is a matching of size 3 avoiding $u_1u_2u_3$.
Hence, we may assume that $e_1 \in M'_2$ and $e_2,e_3 \in M'_1$.
By Claim~\ref{clm:exactconfig}(c)--(e), we determine the exact structure of $L_S(e_i,e_j)$ for $i \ne j$.
Moreover, there is also a matching of size 3, namely $\{v_{1,2} v_{2,1}$, $v_{1,3} v_{3,1}$, $v_{2,3} v_{3,2}\}$.
\end{proof}

\begin{clm} \label{clm:u1U2U3}
The number of edges of type $UUU$ is at most $2 \gamma'' {n'}^3/3$.
\end{clm}

\begin{proof}
Given a vertex $u \in U$, let $e(u)$ be the edge in $M'$ containing $u$.
Since each edge in $M'$ is of type $UUW$, the number of edges $u_1u_2u_3$ of type $UUU$ with $e(u_i)$ not distinct is at most $n' \times 2n'/3 = 2 {n'}^2 /3$.
Let $u_1u_2u_3$ be an edge of type $UUU$ with $e(u_i)$ distinct.
By Claim~\ref{clm:Utriangle}, $\{e(u_1) ,e(u_2), e(u_3)\}$ does not form a triangle in $G'$.
Since $G'$ misses at most $758\gamma \binom{m}2$ edges, at most $127 \gamma m^3$ submatchings $\{e_j, e_{j'}, e_{j''}\} \subseteq M$ do not form a triangle in $G'$.
For each such submatching of size~3, there are at most $8$ ways of choosing $u_1, u_2$ and $u_3$.
Therefore, the number of edges of type $UUU$ in $H'$ is at most $2 {n'}^2 /3 + 8 \times 127 \gamma m^3 \le 2 \gamma'' {n'}^3/3$.
\end{proof}

Finally, we are ready to show that $H'$ is $8\gamma''$-close to $H'_3(n';n')$.
Denote by $e_{H'}[Q_1 Q_2 Q_3]$ the number of edges in $H'$ of type $Q_1 Q_2 Q_3$.
For $u_1 \in U_1$,
\begin{align}
	\deg_{H'}(u_1) & \ge  \delta_1(H) - 2 (n- n') n \ge \left( 5/9 - \gamma - 2 \gamma'' \right) n^2
	\ge \left( 5/9 - 3 \gamma'' \right) {n'}^2 \nonumber.
\end{align}
Hence, there are $( 10/27 - 2 \gamma'' ) {n'}^3$ edges of type $U_1V_2'V_3'$ as $|U_1| = 2n'/3$.
Recall Claim~\ref{clm:u1U2U3} that the number of edges of type $UUU$ is $e_{H'}[UUU] \le 2 \gamma'' {n'}^3/3$.
Therefore,
\begin{align}
	e_{H'}[U_1W_2W_3] +e_{H'}[U_1U_2W_3] + e_{H'}[U_1W_2U_3]& \ge 	(10/27 - 8  \gamma'' /3){n'}^3  \nonumber.
\end{align}
and similar inequalities hold when we swap the indices.
Note that $e_{H'}[UUW] \le 4 {n'}^3 /9$.
Thus,
\begin{align}
	3(10/27 - 8/3 \gamma''){n'}^3 & \le  e_{H'}[UWW] +2 e_{H'}[UUW] \nonumber \\
	&\le  e_{H'}[UWW] +e_{H'}[UUW] +4{n'}^3/9,\nonumber \\
	(2/3 - 8 \gamma''){n'}^3 & \le  e_{H'}[UWW] +e_{H'}[UUW]. \nonumber
\end{align}
Recall that each edge of $H'_3(n';n')$ is of type either $UUW$ or $UWW$.
This means that $e(H'_3(n';n')) = 2{n'}^3/3$ and so $H'$ is $8 \gamma''$-close to $H'_3(n';n')$.
This completes the proof of Lemma~\ref{lma:exact}.
\end{proof}

\section{Extremal Result}\label{sec:extremalresult}

Our aim of this section is to prove that if $H$ is $\varepsilon$-close to $H'_3(n;n)$ and $\delta_1(H) > d_3(n;n-1)$, then $H$ contains a perfect matching. 

\begin{lma} \label{lma:epsilon-close}
For all $0 < \varepsilon < 2^{-2}3^{-22}$, there exists an integer $n'_0$ such that the following holds.
Suppose $H$ is a $3$-partite $3$-graph with each class of size $n \ge n'_0$.
If  $H$ is $\varepsilon$-close to $H'_3(n;n)$ and $\delta_1(H) > d_3(n,n-1)$,
then $H$ contains a perfect matching.
\end{lma}

Our argument follows closely to K{\"u}hn, Osthus, and Treglown~\cite[Lemma~7]{kuhn2010matchings}.
Given a $3$-graph $H$ and a vertex $v \in V(H)$, we write $L^H_v$ for the link graph of $v$ with respect to $H$.
Given $\alpha>0$ and a $3$-partite $3$-graph $H $ on the same partition classes as $H'(n;d_1,d_2,d_3)$, we say a vertex $v \in V(H)$ is \emph{$\alpha$-good with respect to $H(n;d_1,d_2,d_3)$} if $|E(L_{v}^{H(n;d_1,d_2,d_3)} - L^{H}_v)| \le \alpha n^2$.
Otherwise $v$ is said to be \emph{$\alpha$-bad}.
Next we show that if all vertices of $H$ are $\alpha$-good with respect to $H'(n;d_1,d_2,d_3)$ with $d_1 +d_2+d_3 = n$, then $H$ contains a perfect matching.

\begin{lma} \label{lma:alpha-good}
Let $0 < \alpha < 2^{-8}$ and let $n$, $d_1$, $d_2$ and $d_3$ be integers such that $d_1, d_2, d_3 \ge 5n/16$ and $d_1+d_2+d_3 = n \ge 10$.
Suppose that $H$ is a $3$-partite $3$-graph on the same partition classes as $H'(n;d_1,d_2,d_3)$ and every vertex of $H$ is $\alpha$-good with respect to $H'(n;d_1,d_2,d_3)$.
Then $H$ contains a perfect matching.
\end{lma}

\begin{proof}
Let $U_1 \cup W_1$, $U_2 \cup W_2$ and $U_3 \cup W_3$ be the partition classes of $H'(n;d_1,d_2,d_3)$ with $|W_i| = d_i$ for $i \in [3]$.
Consider the largest matching~$M$ in $H$ which consists entirely of edges of type $UUW$, where $U = \bigcup_{i \in [3]} U_i$ and $W = \bigcup_{i \in [3]} W_i$.
For $i \in[3]$, define $M_i$ to be the submatching in $M$ consisting of all edges of type $UUW_i$.
Thus $M = M_1 \cup M_2 \cup M_3$.
Let $U'_i = U_i \setminus V(M)$ and $W'_i = W_i \setminus V(M)$ for $i \in[3]$.
We may assume that $W'_1 \cup W'_2 \cup W'_3 \ne \emptyset$, or else $M$ is a perfect matching in $H$.
Without loss of generality, we may assume that $W'_1 \ne \emptyset$ and so $U'_2 \ne \emptyset$ and $U'_3 \ne \emptyset$.
Next, we are going to show that $|M_1| \ge n/4$.
Let $w_1 \in W'_1$.
Note that $w_1u_2u_3$ is not an edge in $H$ for all $u_i \in U'_i$ and all $i  =2,3$.
Otherwise, $M \cup \{w_1u_2u_3\}$ is a matching contradicting the maximality of~$M$.
Since $w_1$ is $\alpha$-good with respect to $H'(n;d_1,d_2,d_3)$, it follows that 
\begin{align*}
|U'_2||U'_3| \le |E(L_{v}^{H(n;d_1,d_2,d_3)} - L^{H}_v)| \le \alpha n^2
\end{align*}
and so 
$\min \{|U'_2|, |U'_3|\} \le \sqrt{\alpha} n$, say $|U'_2| \le \sqrt{\alpha} n$.
Thus, 
\begin{align*}
|M_1|  
& = |V_2 \cap V(M_1)| 
= |V_2| - |W_2| - |U_2'| - |V_2 \cap V(M_3)| \\
& \ge n - d_2 - \sqrt{\alpha}n - d_3 
= d_1 - \sqrt{\alpha}n \ge n/4
\end{align*}
as claimed.

Pick $w_1 \in W'_1$, $u_2 \in U'_2$ and $u_3 \in U'_3$.
Given a pair of distinct edges $e_1,e_2 \in M_1$, we say that $(e_1,e_2)$ is \emph{good for $w_1u_2u_3$} if $H$ contains all possible edges $e$ such that $V(e) \subseteq V(e_1 \cup e_2) \cup \{w_1,u_2,u_3\}$ and 
\begin{align*}
|V(e) \cap V(e_1)|= |V(e) \cap V(e_2)| = |V(e) \cap \{ w_1,u_2,u_3\}| = 1.
\end{align*}
If $(e_1,e_2)$ is good for $w_1u_2u_3$, then there is a matching $M'$ of size~$3$ in $H$ spanning the vertex set $V(e_1 \cup e_2) \cup \{w_1,u_2,u_3\}$.
Since $V_1 \cap V(M')\subseteq W_1$ and $V(M') \setminus V_1 \subseteq U$, every edge in $M'$ is of type $UUW_1$. 
Therefore, we can obtain a matching, namely $(M - \{e_1,e_2\}) \cup M' $, in $H$ that is larger than $M$ consisting only of edges of type $UUW$, yielding a contradiction.
Hence we may assume that there is no good pair for $w_1u_2u_3$.
Since $|M_1| \ge n/4$, we have at least $\binom{n/4}2 \ge n^2/64$ pairs of distinct edges $e_1,e_2 \in M_1$. Since $w_1$, $u_2$ and $u_3$ are $\alpha$-good with respect to $H'(n;d_1,d_2,d_3)$, there are at most $3 \alpha n^2 < n^2/64$ pairs of distinct edges $e_1,e_2 \in M_1$ such that the pair $(e_1,e_2)$ is not good for $w_1u_2u_3$.
So there exists a pair of edges in $M_1$ that is good for $w_1u_2u_3$, a contradiction.
\end{proof}

Next we prove Lemma~\ref{lma:epsilon-close}.
We are going to consecutively remove 4 matchings $M^1, \dots, M^4$ from~$H$ such that the resulting graph $H^4$ satisfies the hypothesis of Lemma~\ref{lma:alpha-good}.
Therefore there exists a perfect matching $M_5$ in $H^4$ and moreover $M^1 \cup  \dots \cup  M^5$ is a perfect matching in $H$.

\begin{proof}[Proof of Lemma~\ref{lma:epsilon-close}]
Suppose $H$ is as in the statement of the lemma.
Note that $\delta_1(H) > d_3(n,n-1)$.
Let $r$ and $s$ be the unique integers such that $n = rk+s$, $r \ge 0$ and $1 \le s \le 3$.
Hence, $r = \lfloor (n-1)/3 \rfloor$.
Let $U_1 \cup W_1$, $U_2 \cup W_2$, $U_3 \cup W_3$ be the partition classes of $H$ with 
\begin{align}
|W_i| =  \left\lfloor \frac{n+i-1}{3} \right\rfloor = \begin{cases}
	r & \textrm{if $i +s \le 3$,}\\
	r +1 & \textrm{otherwise}
	\end{cases} \nonumber
\end{align}
for $i \in[3]$.
Since $H$ is $\varepsilon$-close to $H'_3(n;n)$, all but at most $3\sqrt{\varepsilon}n$ vertices in $H$ are $\sqrt{\varepsilon}$-good with respect to $H'_3(n;n)$.
If all vertices are $\sqrt{\varepsilon}$-good with respect to $H'_3(n;n)$, then Lemma~\ref{lma:alpha-good} implies that $H$ contains a perfect matching.
Hence, we may assume that there exists one $\sqrt{\varepsilon}$-bad vertex.
Given $i \in [3]$, let $U_i^{\bad}$ be the set of $\sqrt{\varepsilon}$-bad vertices with respect to $H'_3(n;n)$ in $U_i$.
Define $W_i^{\bad}$ similarly.
So $|U_i^{\bad}|, |W_i^{\bad}| \le 3 \sqrt{\varepsilon} n$ for all $i \in [3]$.
Let $c =\max |W_i^{\bad}|$, so $1 \le c \le 3 \sqrt{\varepsilon}n$.
For $i \in [3]$, pick a vertex set $\widetilde{W}_i$ such that $W_i^{\bad} \subseteq \widetilde{W}_i \subseteq W_i$ and 
\begin{align}
|\widetilde{W}_i| = \begin{cases}
	c & \textrm{if  $i +s \le 3$,}\\
	c+1& \textrm{otherwise}.
	\end{cases} \nonumber
\end{align}
Define $U_i^0 = U_i \cup \widetilde{W}_i$ and $W^0_i = W_i \setminus\widetilde{W}_i$.
Note that $|U_i^0| = n - r +c$ and $|W_i^0| = r-c$ for all $i \in [3]$.

We are going to successively remove matchings $M^1, \dots, M^5$ from~$H$.
For convenience, we will use the following notation.
Let $H_0 = H$ and $n_0 = n$.
For $j \in[5]$, let $U^{j}_i = U^{j-1}_i \setminus V(M^j)$, $U^{j} = \bigcup_{i \in [3]}U^{j}_i$, $W^{j}_i  = W^{j-1}_i\setminus V(M^j)$, $W^{j} =  \bigcup_{i \in [3]}W^{j}_i$ and $H^j =  H[W^j \cup U^j] = H^{j-1} \setminus V(M^j)$.
Write $m_j = |M_j|$ and $n_j = |U^{j}_i|+ |W^{j}_i| = n_{j-1}-m_j$.

Recall that $|U_i^0| = n - r +c$.
For each $v \in U_i^0$, there are at most $n^2 - (n-r+c)^2$ edges containing $v$ and one vertex in $W^0$.
Hence, 
\begin{align}
	\delta_1(H[U^0])& >  d_3(n,n-1) - (n^2 - (n-r+c)^2) = d_3(n-r+c,3c+s-1). \nonumber
\end{align}
The last equality holds by the explicit definition of $d_3(n,m)$ given in Section~\ref{sec:introduction}.
Note that $ n-r+c \ge 2n/3 >  3^7(3c+s)$.
Thus, there exists a matching $M^1$ of size $m_1 = 3c+s$ in $H[U^0]$ by Corollary~\ref{cor:partialmatching}.

Let $H' = H'_3(n_1;n_1)$ on the vertex set $V(H^1)$.
Note that if $u \in U^1$ is $\sqrt{\varepsilon}$-good with respect to $H'_3(n;n)$ in $H$, then $u$ is $2\sqrt{\varepsilon}$-good with respect to $H'$ in $H^1$ or else
\begin{align}
	\sqrt{\varepsilon} n^2 \ge |E(L_{u}^{H_3'(n;n)} - L^{H}_u)| \ge  |E(L_{u}^{H'} - L^{H^1}_u)| > 2\sqrt{\varepsilon} n_1^2\ge \sqrt{\varepsilon} n^2, \nonumber
\end{align}
a contradiction.
If $w \in W_1 \cap V(H^1)$ is $\sqrt{\varepsilon}$-good with respect to $H'_3(n;n)$ in $H$, then 
\begin{align}
	 |E(L_{w}^{H'} - L^{H^1}_w)| & \le  |E(L_{w}^{H'_3(n;n)} - L^{H}_w)| + |\widetilde{W}_2||W^1_3|+|\widetilde{W}_3||W^1_2| \nonumber	\\
	& \le \sqrt{\varepsilon}n^2  +  c n^2/3  + c n^2/3 < 5 \sqrt{\varepsilon} n_1^2. \nonumber
\end{align}
Therefore, if a vertex $v \in V(H^1)$ is $5 \sqrt{\varepsilon}$-bad with respect to~$H'$, then $v$ is $\sqrt{\eps}$-bad with respect to $H'_3(n,n)$ in $H$.
Define $U^{1,\bad}$ to be the set of such vertices.
So $U^{1,\bad} \subseteq \bigcup_{i \in [3]} (U_i^{\bad} \cup {W}_i^{\bad}) \subseteq \bigcup_{i \in [3]} (U_i^{\bad} \cup \widetilde{W}_i)$ and $|U^{1,\bad}| \le 3 \sqrt{\varepsilon}n$.
If $U^{1,\bad} = \emptyset$, then there exists a matching $M'_2$ of size $n_1 = n -m_1$ in $H^1$ by Lemma~\ref{lma:alpha-good}.
Thus, $M_1 \cup M'_2$ is a perfect matching in $H$.
So we may assume that $U^{1,\bad} \ne \emptyset$.
A vertex $u \in U^{1,\bad}$ is \emph{useful} if there exist greater than $6\sqrt{\varepsilon} n^2$ pairs of vertices $(u',w) \in U^1 \times W^1$ such that $uu'w$ is an edge in $H^1 \subseteq H$.
Clearly we can greedily select a matching $M^2$ in $H^1$ of size $m_2 \le |U^{1,\bad}| \le 3\sqrt{\varepsilon} n$, where $M^2$ covers all useful vertices and consists entirely of edges of type $U^1U^1W^1$.

Pick $u \in U^{1,\bad} \cap U^2$, say $u \in V_1$.
Note that $u$ is not useful and so the number of edges in $H^2[U^2]$ containing $u$ is 
\begin{align*}
\deg_{H^2[U^2] }(u) & \ge  \deg_{H^2}(u) - 6 \sqrt{\eps} n^2  - |W^2_2||W^3_3| \\
& \ge \delta_1(H) - 2 n (m_1 + m_2) -  6 \sqrt{\eps} n^2  - \lceil n/3 \rceil ^2
 \ge 6 \sqrt{\varepsilon} n^2.
\end{align*}
Since $|U^{1,\bad}| \le 3 \sqrt{\varepsilon}n$, once again, we can greedily select a matching $M^3$ in $H^{2}[U^2]$ of size $|U^{1,\bad} \cap U^2|$ such that $M^3$ covers all $U^{1,\bad} \cap U^2$.

Note that $H^3$ is a $3$-partite $3$-graph with each partition of size $n_3 = n - m_1-m_2-m_3 \ge (1- 20\sqrt{\varepsilon})n$.
Let $d_i = |W^3_i|$ for $i \in[3]$.
Our aim is to show that every vertex in $H^3$ is $6\sqrt{\varepsilon}$-good with respect to $H'(n_3;d_1,d_2,d_3)$.
Note that $V(H^3) \subseteq V(H^1) \setminus U^{1,\bad}$ and $H'(n_3;d_1,d_2,d_3) \subseteq H'$.
So for every vertex $v \in V(H^3)$, we have 
\begin{align*}
	 |E(L_{v}^{H^3} - L^{H'(n_3;d_1,d_2,d_3)}_v)| & \le  |E(L_{v}^{H^1} - L^{H'}_v)| \le 5 \sqrt{\eps} n_1^2 \le 6 \sqrt{\eps} n_1^3.
\end{align*}
Therefore, every vertex in $H^3$ is $6 \sqrt{\varepsilon}$-good with respect to $H'(n_3;d_1,d_2,d_3)$ as claimed.
Recall that $H'(n_3;d_1,d_2,d_3)$ contains all edges of type $U^4U^4W^4$ and $U^4W^4W^4$.
Therefore, we can greedily find a matching $M^4$ of size $m_4 = m_1+m_3 \le 13 \sqrt{\varepsilon}n$ edges consisting of edges of type $U^4W^4W^4$.

Finally, notice that all vertices of $H^4$ are $8\sqrt{\varepsilon}$-good with respect to $H'(n_4;d_1',d_2',d_3')$, where $d_i' = |W^4_i|$ for all $i \in[3]$.
Note that $d_i' \ge 5n_4/16$ for each $i \in [3]$ and 
\begin{align}
d_1'+d_2'+d_3' = n - m_2 - 2 m_4 = n - \sum_{i \in [4]} m_i = n_4. \nonumber
\end{align}
Thus, we can apply Lemma~\ref{lma:alpha-good} to obtain a perfect matching $M^5$ in $H^4$, so $M^1 \cup \dots \cup M^5$ is a perfect matching in~$H$ as required.
\end{proof}

Finally, we are ready to prove Theorem~\ref{thm:exact}. 

\begin{proof}[Proof of Theorem~\ref{thm:exact}]
Let $\gamma>0$ be sufficiently small and let $n_0$ be a sufficiently large integer.
Set $\gamma' = \gamma^5/20$ and $\gamma'' = 3(1200 \gamma + \sqrt{\gamma/2})$.
Since $\gamma$ is small, $H$ satisfies the hypothesis of Lemma~\ref{lma:absorptionlemma} with $k = 3$ and $l=1$.
Let $M$ be the matching given by Lemma~\ref{lma:absorptionlemma} and so $|M| \le 2 \gamma^3 n$.
Let $H' = H \setminus V(M)$.
Note that 
\begin{align}
	\delta_1(H') \ge \delta_1(H)-  2 |M| n \ge \left( 5/9- \gamma   \right){n'}^2,
\nonumber
\end{align}
where $n' = n - |M|$.
First suppose that there exists a matching $M'$ in $H'$ covering all but at most $\gamma' n'$ vertices in each partition class.
Let $W = V(H) \setminus V(M \cup M')$.
Clearly, $W$ is balanced and has size at most $3 \gamma' n$.
By the property of $M$, there exists a matching $M''$ covering exactly the vertices of $V(M) \cup W$.
Thus, $M'$ and $M''$ form a perfect matching.

Therefore, we may assume that there is no matching in $H'$ of size $(1-\gamma')n'$.
By Lemma~\ref{lma:exact} taking $\rho = \gamma'$, $H = H'$ and $n = n'$, there exists a subgraph $H''$ in $H'$ such that $H''$ is $8 \gamma''$-close to $H'_3(n'';n'')$ such that  $n'' \ge  (1- \gamma'')n' \ge(1-2 \gamma'') n $ and $ 3 | n''$.
Note that
\begin{align*}
	|E(H'_3(n;n)-H)| & \le |E(H'_3(n;n) -H'')| \nonumber \\
	& \le  8 \gamma'' n''^3 + 3(n-n'')n^2 \le 14 \gamma'' n^3,
\end{align*}
so $H$ is $14 \gamma''$-close to $H'_3(n;n)$.
Therefore $H$ has a perfect matching by Lemma~\ref{lma:epsilon-close}.
\end{proof}

\section{Acknowledgment}

The authors would like to thank Victor Falgas for his helpful comments and to the referees for their careful reading, one of whom in particular read the paper extremely carefully and suggested helpful clarifications.

\end{document}